\documentclass[12pt,reqno]{amsart}

\usepackage{amssymb,amsmath, amsfonts, latexsym}
\usepackage{enumerate}

\usepackage{graphics,graphicx}
\usepackage [dvips]{epsfig}

\newcommand{\dE}{\mathbb{E}}

\newcommand{\dR}{\mathbb{R}}
\newcommand{\dP}{\mathbb{P}}

\newcommand{\cF}{\mathcal{F}}

\newcommand{\veps}{\varepsilon}

\def\beq{\begin{equation}}

\def\bdef{\begin{defn}}
\def\ndef{\end{defn}}

\def\bthm{\begin{thm}}
\def\nthm{\end{thm}}

\def\bprop{\begin{prop}}
\def\nprop{\end{prop}}

\def\brmk{\begin{remarks}}
\def\nrmk{\end{remarks}}

\def\bexa{\begin{example}}
\def\nexa{\end{example}}

\def\blem{\begin{lem}}
\def\nlem{\end{lem}}

\def\bcor{\begin{cor}}
\def\ncor{\end{cor}}

\def\bexe{\begin{exe}}
\def\nexe{\end{exe}}

\def\bprf{\begin{proof}}
\def\nprf{\end{proof}}

\def\bnum{\begin{enumerate}}
\def\nnum{\end{enumerate}}

\newtheorem{thm}{Theorem}[section]
\newtheorem{cor}[thm]{Corollary}
\newtheorem{lem}[thm]{Lemma}
\newtheorem{prop}[thm]{Proposition}
\newtheorem{example}[thm]{Example}
\newtheorem{remarks}[thm]{Remarks}
\newtheorem{rem}[thm]{Remark}
\newtheorem{defn}[thm]{Definition}

\def\FF{\mathcal F}
\def\GG{\mathcal G}

\def\MM{\mathcal M}

\newcommand{\ee}{\mathbb{E}}

\newcommand{\nn}{\mathbb{N}}
\newcommand{\rr}{\mathbb{R}}
\newcommand{\pp}{\mathbb{P}}

\newcommand{\ttt}{\mathbb{T}}

\newcommand{\g}{\mathbb{G}}
\newcommand{\ff}{\mathbb{F}}
\newcommand{\xx}{\mathbb{X}}
\newcommand{\yy}{\mathbb{Y}}
\newcommand{\uu}{\mathbb{U}}
\def\vep{\varepsilon}
\def\<{\langle}
\def\>{\rangle}

\newcommand{\superexpp} {\underset{b_n^2}{\overset{\rm superexp}{\Longrightarrow}} }

\newcommand{\superexp} {\underset{b_{|\mathbb{T}_{n-1}|}^2}{\overset{\rm superexp}{\Longrightarrow}} }
\newcommand{\superexpn} {\underset{b_{|\mathbb{T}_{n}|}^2}{\overset{\rm superexp}{\Longrightarrow}} }
\newcommand{\superexpnn} {\underset{b_{n}^2}{\overset{\rm superexp}{\Longrightarrow}} }

\def\build#1_#2^#3{\mathrel{\mathop{\kern 0pt#1}\limits_{#2}^{#3}}}

 \numberwithin{equation}{section}

\begin{document}

\title[Deviation for  bifurcating autoregressive processes]{Deviation inequalities and moderate deviations for estimators of parameters in bifurcating autoregressive models }

\author{S.Val\`ere Bitseki Penda}
\email{Valere.Bitsekipenda@math.univ-bpclermont.fr}
\address{Laboratoire de Math\'ematiques, CNRS UMR 6620, Universit\'e Blaise Pascal, Avenue des Landais, 63177 Aubi\`ere, France.}

\author{Hac\`ene Djellout}
\email{Hacene.Djellout@math.univ-bpclermont.fr}
\address{Laboratoire de Math\'ematiques, CNRS UMR 6620, Universit\'e Blaise Pascal, Avenue des Landais, 63177 Aubi\`ere, France.}

\keywords{Deviation inequalities, Moderate deviation principle, Bifurcating autoregressive process,
  Martingale, Limit theorems, Least squares estimation.}

\date{\today}
\begin{abstract} The purpose of this paper is to investigate the deviation inequalities and the moderate
  deviation principle of the least squares estimators of the unknown
  parameters of general $p$th-order bifurcating autoregressive processes, under suitable assumptions on the driven noise of the process. Our investigation
  relies on the moderate deviation principle for martingales.
\end{abstract}

\maketitle

\vspace{-0.5cm}

\begin{center}
\textit{AMS 2000 subject classifications: 60F10, 62F12, 60G42, 62M10, 62G05.}
\end{center}

\medskip

\section{Motivation and context }

Bifurcating autoregressive processes (BAR, for short) are an
adaptation of autoregressive processes, when the data have a binary
tree structure. They were first introduced by Cowan and Staudte
\cite{CS86}   for cell lineage data where each individual in one
generation gives rise to two offspring in the next generation.
\medskip

In their paper, the original BAR process was defined as follows. The
initial cell is labelled $1$, and the two offspring of cell $k$ are
labelled $2k$ and $2k+1$. If $X_k$ denotes an observation of some
characteristic of individual $k$ then the first order BAR process is
given, for all $k\ge 1$, by
\begin{equation*}\begin{cases}
X_{2k}=a +bX_k+\vep_{2k} \\
X_{2k+1}=a+b X_k+\vep_{2k+1}.
\end{cases}
\end{equation*}
The noise sequence $(\varepsilon_{2k},\varepsilon_{2k+1})$
represents environmental effects, while $a,b$ are unknown real
parameters, with $|b|<1$, related to inherited effects. The driven
noise $(\varepsilon_{2k},\varepsilon_{2k+1})$ was originally
supposed to be independent and identically distributed with normal
distribution.  But since two sister cells are in the same
environment at their birth, $\varepsilon_{2k}$ and
$\varepsilon_{2k+1}$ are allowed to be correlated, inducing a
correlation  between sister cells, distinct from the correlation
inherited from their mother.
\medskip

Several extensions of the model have been proposed and various
estimators are studied in the literature for the unknown parameters,
see for instance \cite{BZ4},\cite{BH99}, \cite{BH00}, \cite{BHY09},
\cite{BZ105}, \cite{BZ205}. See  \cite{BSG09} for a relevant
references.
\medskip

Recently, there are many studies of the asymmetric BAR process, that
is when the quantitative characteristics of the even and odd sisters
are allowed to depend from their mother's through different sets of
parameters.

\medskip
Guyon \cite{G07} proposes an interpretation of the asymmetric BAR
process as a bifurcating Markov chain, which allows him to derive
laws of large numbers and central limit theorems for the least
squares estimators of the unknown parameters of the process.  This
Markov chain approach was further developed by  Delmas and Marsalle
\cite{DM10}, where the cells are allowed to die. They defined the
genealogy of the cells through a Galton-Watson process, studying the
same model on the Galton Watson tree instead of a binary tree.

\medskip
Another approach based on martingales theory was proposed by Bercu,
de Saporta and G\'egout-Petit \cite{BSG09}, to sharpen the asymptotic analysis of Guyon
under weaker assumptions. It must be pointed out that missing data are not
dealt with in this work. To take into account possibly missing data in the estimation procedure de Saporta et al.
\cite{SGM11} use a two-type Galton-Watson process to
model the genealogy.
\medskip

Our objective  in this paper  is to go a step further by
\begin{itemize}
\item studying  the moderate deviation principle (MDP, for short) of the least squares estimators of the unknown
parameters of general $p$th-order bifurcating autoregressive
processes.  More precisely we are interested in the asymptotic
estimations of
$$\pp\left(\frac{\sqrt{n}}{b_n}\big(\Theta_n-\Theta\big)\in A\right)$$
where $\Theta_n$ denotes the  estimator  of the unknown parameter of
interest $\Theta$, $A$ is a given domain of deviation, $(b_n>0)$ is
some sequence denoting the scale of deviation. When $b_n=1$ this
exactly the estimation of the central limit theorem. When $b_n=\sqrt
n$, it becomes the {\it large deviation}. And when $1\ll b_n\ll\sqrt
n$, this is the so called {\it moderate deviations}. Usually, MDP
has a simpler rate function inherited from the approximated Gaussian
process, and holds for a larger class of dependent random variables
than the LDP.
\medskip

Though we have not found studies exactly on this question in the
literatures, except the recent work of Biteski et al. \cite{BDG11}
but technically we are much inspired from two lines of studies

\begin{enumerate}
\item the work of Bercu et al. \cite{BSG09} on the almost sure convergence of the estimators with
the quadratic strong law and the central limit theorem;
\item the works of Dembo \cite{Dem96}, and Worms \cite{Wor01a}, \cite{Wor01b},
  \cite{Wor01c} on the one hand, and of the paper of Puhalskii \cite{Puha97}
  and Djellout \cite{Dj02} on the other hand, about the MDP for martingales.

\end{enumerate}
\item giving deviation inequalities for the estimator of bifurcating autoregressive processes,
which are important for a rigorous non asymptotic statistical study, i.e. for all
$x>0$
$$\pp\left(||\Theta_n-\Theta||\ge x\right)\le e^{-C_n(x)},$$ where
$C_n(x)$ will crucially depends on our set of assumptions. The upper
bounds in this inequality hold for arbitrary $n$ and $x$ (not a
limit relation, unlike the MDP results), hence they are much more
practical (in statistics). Deviation inequalities for estimators of
the parameters  associated  with linear regression,  autoregressive
and branching processes are investigated by Bercu and Touati
\cite{BeTo08}. In the martingale case, deviation inequalities for
self normalized martingale have been developed by de la
Pe$\overset{\sim}{\rm n}$a et al. \cite{PeTzQu09}.  We also refer to
the work of Ledoux \cite{Le01} for precise credit and references.
This type of inequalities are equally well motivated by theoretical
question as by numerous applications in different field including
the analysis of algorithms, mathematical physics and empirical
processes. For some applications  in non asymptotic model selection
problem we refer to Massart \cite{Ma07}.
\end{itemize}
\medskip

This paper is organized as follows. First of all, in Section 2, we introduce
the BAR($p$) model as well as the least square estimators for the parameters
of observed BAR($p$) process and some related notation and hypothesis. In
Section 3, we state our main results on the deviation inequalities and MDP of our estimators.
The section 4 dedicated to the superexponential convergence of the quadratic variation of the
martingale, this section contains exponential inequalities which are crucial for the proof of the
deviation inequalities. The proofs of the main results are postponed in section 5.


\section{Notations and Hypothesis}
In all the sequel, let $p\in \nn^*$. We consider the asymmetric
BAR($p$) process given, for all $n\ge 2^{p-1}$, by
\begin{equation}\label{bar_p1}\begin{cases}X_{2n}=a_0+\sum_{k=1}^pa_kX_{[\frac{n}{2^{k-1}}]}+\vep_{2n} \\
X_{2n+1}=b_0+\sum_{k=1}^pb_kX_{[\frac{n}{2^{k-1}}]}+\vep_{2n+1},
\end{cases}
\end{equation}
where the notation $[x]$ stands for the largest integer less than or equal to the real
$x$.
The initial states $\{X_k,1\le k\le 2^{p-1}-1\}$ are the
ancestors while $(\vep_{2n},\vep_{2n+1})$ is the driven noise of the
process. The parameters $(a_0,a_1,\cdots,a_p)$ and
$(b_0,b_1,\cdots,b_p)$ are unknown real numbers.

The BAR($p$) process can be rewritten in the abbreviated vector form given, for all $n\ge
2^{p-1}$, by \begin{equation}\label{bar_p2}\begin{cases}\xx_{2n}=A\xx_n +\eta_{2n}\\
 \xx_{2n+1}=B\xx_n +\eta_{2n+1}
\end{cases}
\end{equation}
where the regression vector
$\xx_n=\left(X_n,X_{[\frac{n}{2}]},\cdots,X_{[\frac{n}{2^{p-1}}]}\right)^t$,
$\eta_{2n}=(a_0+\vep_{2n})e_1$, $\eta_{2n+1}=(b_0+\vep_{2n+1})e_1$,
with $e_1=(1,0,\cdots,0)^t\in\rr^p$. Moreover, $A$ and $B$ are the
$p\times p$ companion matrices
$$A=\begin{pmatrix}
a_1& a_2 & \cdots & a_p \\
1 & 0 & \cdots & 0 \\
0 & . & . & .\\
0& . & 1& .\\
\end{pmatrix}
\qquad {\rm and}\quad B=\begin{pmatrix}
b_1& b_2 & \cdots & b_p \\
1 & 0 & \cdots & 0 \\
0 & . & . & .\\
0& . & 1& .\\
\end{pmatrix}.
$$
In the sequel, we shall assume that the matrices $A$ and $B$ satisfy the
contraction property
\begin{equation}\label{beta}\beta=\max(|| A||, ||B ||)<1,\end{equation}
where  for any matrix $M$ the notation $M^t$, $\Vert M \Vert$ and ${\rm Tr}(M)$ stand for the
transpose, the euclidean norm and the trace of $M$, respectively.

On can see this BAR($p$) process as a $p$th-order autoregressive
process on a binary tree, where each vertex represents an individual
or cell, vertex $1$  being the original ancestor. For all $n\ge 1$,
denote the $n$-th generation by $\displaystyle \g_n=\{2^n,2^n+1,\cdots,2^{n+1}-1\}.$
\begin{figure}[bht]
\includegraphics[width=4.83in,height=4.38in]{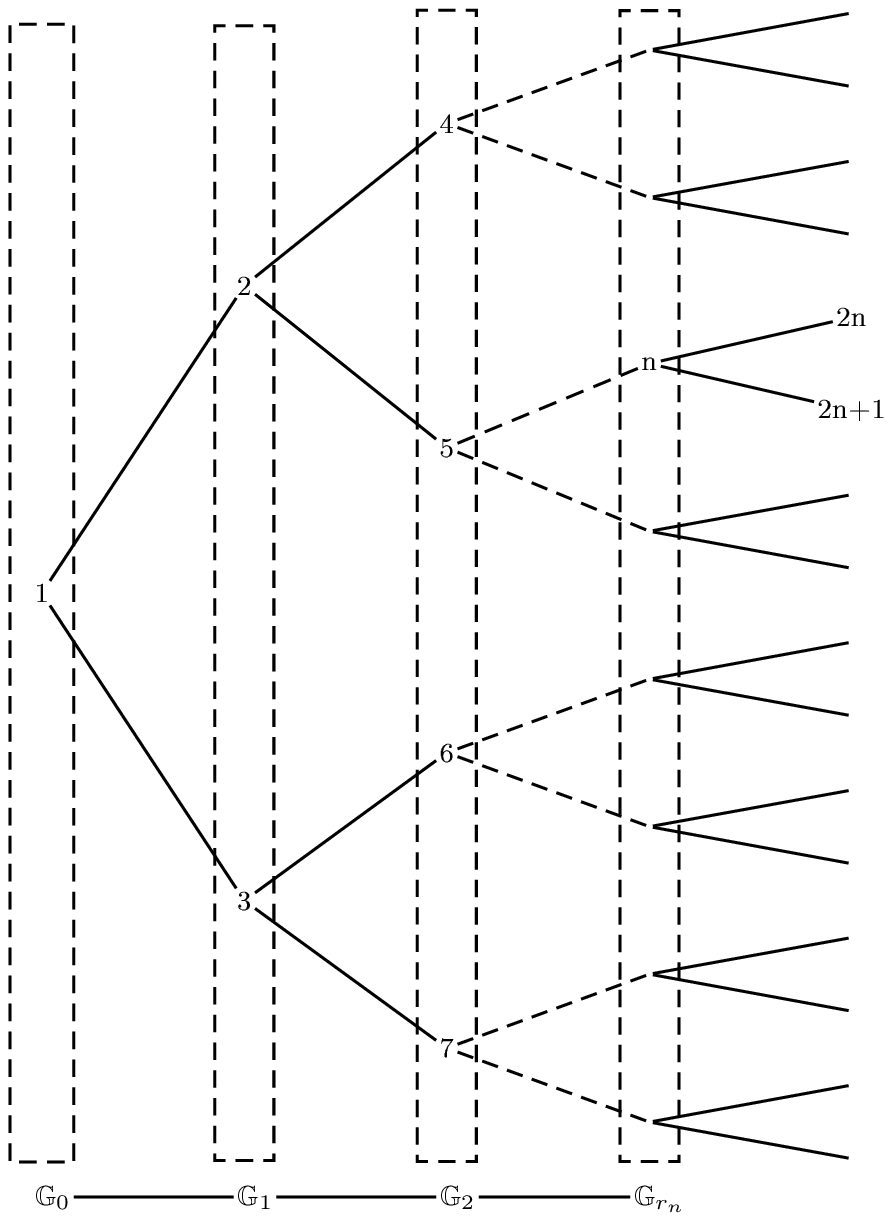}
\caption{The binary tree $\mathbb{T}$}\label{Fi:arbrebinaire}
\end{figure}

In particular, $\g_0=\{1\}$ is the initial generation and
$\g_1=\{2,3\}$ is the first generation of offspring from the first
ancestor. Let $\g_{r_n}$ be the generation of individual $n$, which
means that $r_n=[\log_2(n)]$. Recall that the two offspring of
individual $n$ are labelled $2n$ and $2n+1$, or conversely, the
mother of the individual $n$ is $[n/2]$. More generally, the
ancestors of individual $n$ are $[n/2],[n/2^2],Ê\cdots,[n/2^{r_n}]$.
Furthermore, denote by
$$\ttt_n=\bigcup_{k=0}^n\g_k$$
the subtree of all individuals from the original individual up to
the $n$-th generation. We denote by $\ttt_{n,p}=\{k\in \ttt_n,k\ge
2^p\}$ the subtree of all individuals up to the $n$th generation
without $\ttt_{p-1}$. One can observe that , for all $n\ge 1$,
$\ttt_{n,0}=\ttt_n$ and for all $p\ge 1$, $\ttt_{p,p}=\g_p$.

The BAR($p$) process can be rewritten, for all $n\ge 2^{p-1}$, in the
matrix form
$$Z_n=\theta^tY_n+V_n$$
where
$$Z_n=\begin{pmatrix}
X_{2n}\\
X_{2n+1}
\end{pmatrix},\quad Y_n=\begin{pmatrix}
1\\
\xx_n
\end{pmatrix},\quad V_n=\begin{pmatrix}
\vep_{2n}\\
\vep_{2n+1}
\end{pmatrix},$$
and the $(p+1)\times 2$ matrix parameter $\theta$ is given by
$$\theta=\begin{pmatrix}
a_0& b_0\\
a_1 & b_1\\
.&.\\
.&.\\
a_p &b_p
\end{pmatrix}.$$

As in Bercu et al.\cite{BSG09}, we introduce the least square estimator $\hat
\theta_n$ of $\theta$, from the observation of all
individuals up to the $n$-the generation that is the complete
sub-tree $\ttt_n$,  for all $n\ge p$
\begin{equation}\label{hattheta}\hat\theta_n=S_{n-1}^{-1}\sum_{k\in \ttt_{n-1,p-1}}Y_kZ_k^t,\end{equation}

where the $(p+1)\times (p+1)$ matrix is defined as
\begin{equation}\label{defS_n}S_n=\sum_{k\in \ttt_{n,p-1}}Y_kY_k^t=\sum_{k\in \ttt_{n,p-1}}\begin{pmatrix}
1& \xx_k^t\\
\xx_k & \xx_k\xx_k^t\\
\end{pmatrix}.
\end{equation}
We assume, without loss of generality, that for all $n\ge p-1$, $S_n$
is invertible. In all what follows, we shall make a slight abuse of notation by
identifying $\theta$ as well as $\hat\theta_n$ to
$${\rm vec}(\theta)=\begin{pmatrix}
a_0\\
.\\
.\\
a_p\\
b_0\\
.\\
.\\
b_p
\end{pmatrix} \quad{\rm and}\quad {\rm vec}(\hat\theta_n)=\begin{pmatrix}
\hat a_{0,n}\\
.\\
.\\
\hat a_{p,n}\\
\hat b_{0,n}\\
.\\
.\\
\hat b_{p,n}\end{pmatrix} .$$

Let $\Sigma_n=I_2\otimes S_n$, where $\otimes$ stands for the matrix Kronecker
product. Therefore, we deduce from (\ref{hattheta}) that
\begin{equation}\label{hatheta}\aligned\hat\theta_n&=\Sigma_{n-1}^{-1}\sum_{k\in \ttt_{n-1,p-1}}{\rm vec}(Y_kZ_k^t)=\Sigma_{n-1}^{-1}\sum_{k\in \ttt_{n-1,p-1}}\begin{pmatrix}
X_{2k}\\
X_k\xx_{2k}\\
X_{2k+1}\\
X_k\xx_{2k+1}
\end{pmatrix}.
\endaligned
\end{equation}
Consequently, (\ref{bar_p2}) yields to
\begin{equation}\label{thest}\hat\theta_n-\theta=\Sigma^{-1}_{n-1}\sum_{k\in \ttt_{n-1,p-1}}\begin{pmatrix}
\vep_{2k}\\
\vep_{2k}\xx_k\\
\vep_{2k+1}\\
\vep_{2k+1}\xx_k
\end{pmatrix}.
\end{equation}
\medskip

Denote by $\ff=(\FF_n)$ the natural filtration associated with the
BAR($p$) process, which means that $\FF_n$ is the $\sigma-$algebra
generated by the individuals up to $n$-$th$ generation, in other words $\FF_n=\sigma\{X_k,k\in \ttt_n\}.$

For the initial states, if we denote by
$\overline{X}_{1}=\max\Big\{\|\xx_{k}\|,k\leq 2^{p-1}\Big\},$ we
introduce the following hypothesis
\begin{enumerate}
\item[{\bf (Xa)}] For some $a>2$, there exists $\tau>0$ such that
\[\ee\left[\exp\left(\tau\overline{X}_{1}^{a}\right)\right]<\infty.
\]
\end{enumerate}
This assumption implies the weaker Gaussian integrability condition
\begin{itemize}
\item[{\bf (X2)}] There is $\tau>0$ such that
\[
\ee\left[\exp\left(\tau\overline{X}_{1}^{2}\right)\right]<\infty.
\]
\end{itemize}

For the noise $(\varepsilon_{2n},\varepsilon_{2n+1})$ the assumption may be of two types.\\
\begin{enumerate}
\item In the first case we will assume the independence of the noise which allows us to impose less restrictive conditions on the exponential integrability of the noise.\\

{\bf Case 1:}  We shall assume that $((\varepsilon_{2n},
\varepsilon_{2n+1} ),n\ge 1)$ forms a sequence of independent and
identically distributed bi-variate centered random variables with
covariance matrix $\Gamma$ associated with
$(\vep_{2n},\vep_{2n+1})$, given by
\begin{equation}\label{Gamma}\Gamma=\begin{pmatrix}
\sigma^2 & \rho\\
\rho & \sigma^2
\end{pmatrix}, \quad \text{where $\sigma^{2}>0$ and $|\rho|<\sigma^{2}.$}
\end{equation}
For all $n\ge p-1$ and for all $k\in \g_{n}$, we denote
\[\ee[\vep_k^2]=\sigma^2,\,\,\, \ee[\vep_k^4]=\tau^4,\,\,\, \ee[\vep_{2k}\vep_{2k+1}]=\rho,\,\,\,\ee[\vep_{2k}^2\vep^2_{2k+1}]=\nu^2\,\,\,{\rm where}\,\, \tau^{4} >0, \nu^{2}<\tau^{4}.\]

In addition, we assume that the condition {\bf (X2)} on the initial state is satisfied  and
\begin{enumerate}
\item[{\bf (G2)}] one can find $\gamma>0$ and $c>0$ such
that for all $n\geq p-1,$ for all $k\in \g_{n}$ and for all
$|t|\le c$
\[
\ee\left[\exp t\left(\vep_{k}^{2} - \sigma^{2}\right)\right]\leq
\exp\left(\frac{\gamma t^{2}}{2}\right).
\]
\end{enumerate}
In this case, we impose the following hypothesis on the scale of the
deviation

\begin{enumerate}
\item[{\bf (V1)}] $(b_{n})$ will denote an
increasing sequence of positive real numbers such that
$$b_{n}\longrightarrow +\infty$$
and for $\beta$ given by (\ref{beta})
\begin{itemize}
\item if $\beta\leq\frac{1}{2}$, the sequence $(b_{n})$ is such that $\displaystyle \frac{b_{n}\log n}{\sqrt{n}}
\longrightarrow 0$,

\item  if $\beta>\frac{1}{2}$, the sequence $(b_{n})$ is such that $\displaystyle
(b_{n}\sqrt{\log n})\beta^{\frac{r_{n}+1}{2}} \longrightarrow 0.$
\end{itemize}
\end{enumerate}

\vskip 15pt

\item In contrast with the first case,  in the second case, we will not assume that the sequence
$((\varepsilon_{2n}, \varepsilon_{2n+1} ),n\ge 1)$ is  i.i.d. The price to pay for giving up this i.i.d.
assumption is higher exponential moments. Indeed we need them to make use of the MDP for martingale,
especially to prove the Lindeberg condition via Lyapunov one's.

\vskip 15pt

{\bf Case 2:}  We shall assume that for all $n\ge p-1$ and for all
$j\in \g_{n+1}$ that  $\ee[\vep_{j}/\FF_{n}]=0$ and for all different
$k,l\in \g_{n+1}$ with $[\frac{k}{2}]\neq [\frac{l}{2}]$, $\vep_{k}$
and $\vep_{l}$ are conditionally independent given $\FF_{n}.$ And we will use the same  notations
as in the case 1:  for all $n\ge p-1$ and for all $k\in \g_{n+1}$
\[
\ee[\vep_k^2/\FF_n]=\sigma^2,\,\,\, \ee[\vep_k^4/\FF_n]=\tau^4,\,\,\, \ee[\vep_{2k}\vep_{2k+1}/\FF_n]=\rho,\,\,\,\ee[\vep_{2k}^2\vep^2_{2k+1}/\FF_n]=\nu^2\,\,\, a.s.
\]
where $\tau^{4} >0$, $\nu^{2}<\tau^{4}$ and we use also  $\Gamma$
for the conditional covariance matrix associated with
$(\vep_{2n},\vep_{2n+1})$. In this case, we assume that the condition {\bf (Xa)} on the initial state is satisfied,
and we shall make use of the following  hypotheses:

\begin{itemize}
\item [{\bf (Ea)}] for some $a>2$, there exist $t>0$ and $E>0$ such that for all $n\geq p-1$ and for all $k\in \g_{n+1},$
\[
\ee\left[\exp\left(t|\vep_{k}|^{2a}\right)/\FF_n\right]\leq E<\infty\,\,\quad a.s.
\]
\end{itemize}


Throughout this case, we introduce the following hypothesis on the scale of the deviation

\begin{enumerate}
\item[{\bf (V2)}] $(b_{n})$ will denote an increasing sequence of positive real numbers such that
$$b_{n}\longrightarrow +\infty,$$
and for $\beta$ given by (\ref{beta})
\begin{itemize}
\item if $\beta^{2}<\frac{1}{2}$, the sequence $(b_{n})$ is such that $\displaystyle \frac{b_{n}\log n}{\sqrt{n}}
\longrightarrow 0$,

\item if $\beta^{2}=\frac{1}{2}$, the sequence $(b_{n})$ is such that $\displaystyle \frac{b_{n}(\log
n)^{3/2}}{\sqrt{n}}\longrightarrow 0$,

\item  if $\beta^{2}>\frac{1}{2}$, the sequence $(b_{n})$ is such that $\displaystyle
(b_{n}\log n)\beta^{r_{n}+1} \longrightarrow 0.$
\end{itemize}
\end{enumerate}
\end{enumerate}

\brmk
The condition on the scale of the deviation in the case 2, is
less restrictive than in the case 1, since we assume more
integrability conditions. This condition on the scale of the
deviation naturally appear from the calculations (see the proof of
Proposition \ref{convergence_crochet}). Specifically, the $\log$
term comes from the crossing of the probability of a sum to the sum
of probability.
\nrmk

\brmk From \cite{DjGuWu06} or \cite{Le01}, we deduce with {\bf (Ea)}  that
\begin{itemize}
\item[{\bf (N1)}]  there is $\phi>0$ such that for all $n\ge p-1$,
for all $k\in \g_{n+1}$ and for all $t\in \rr$,
\[
\ee\Big[\exp\left(t \vep_k\right)/{\FF_n}\Big]< \exp\left(\frac{\phi
t^2}{2}\right),\qquad a.s.
\]
\end{itemize}
We have the same conclusion in the case 1, without the conditioning
; i.e.
\begin{itemize}
\item[{\bf (G1)}]  there is $\phi>0$ such that for all $n\ge p-1$,
for all $k\in \g_{n}$ and for all $t\in \rr$,
\[\ee\Big[\exp(t \vep_k)\Big]< \exp\left(\frac{\phi
t^2}{2}\right).\]
\end{itemize}
\nrmk

\brmk Armed by the recent development in the theory of
transportation inequalities, exponential integrability and
functional inequalities (see Ledoux \cite{Le01}, Gozlan \cite{Go06}
and Gozlan and Leonard \cite{GoLe07}), we can prove that a
sufficient condition for hypothesis {\bf (G2)} to hold is existence
of $t_{0}>0$ such that for all $n\geq p-1$ and for all $k\in
\g_{n}$, $\dE\left[\exp(t_{0}\vep_{k}^{2})\right]<\infty.$ \nrmk

\bigskip
We now turn to the estimation of the parameters $\sigma^2$ and
$\rho$. On the one hand, we propose to estimate the conditional
variance $\sigma^2$ by
\begin{equation}\label{hatsigma_n}
\hat\sigma_n^2=\frac{1}{2|\ttt_{n-1}|}\sum_{k\in \ttt_{n-1,p-1}}||\hat V_k||^2=\frac{1}{2|\ttt_{n-1}|}
\sum_{k\in \ttt_{n-1,p-1}}(\hat\vep_{2k}^2+\hat\vep_{2k+1}^2)
\end{equation}
where for all $n\ge p-1$ and all $k\in \g_n$ , $\hat
V_k^t=(\hat\vep_{2k},\hat\vep_{2k+1})^t$ with
\[
\begin{cases}\hat\vep_{2k}=X_{2k}-\hat a_{0,n}-\sum_{i=1}^p\hat a_{i,n}X_{[\frac{k}{2^{i-1}}]}\\
 \hat\vep_{2k+1}=X_{2k+1}-\hat b_{0,n}-\sum_{i=1}^p\hat b_{i,n}X_{[\frac{k}{2^{i-1}}]}\
\end{cases}
\]
We also introduce the following
\begin{equation}\label{sigma_n}
\sigma^2_n=\frac{1}{2|\ttt_{n-1}|}\sum_{k\in\ttt_{n-1,p}}(\vep_{2k}^2+\vep_{2k+1}^2).
\end{equation}

On the other hand, we estimate the conditional covariance $\rho$ by
\begin{equation}\label{hatrho_n}
\hat\rho_n=\frac{1}{|\ttt_{n-1}|}\sum_{k\in
\ttt_{n-1,p-1}}\hat\vep_{2k}\hat\vep_{2k+1}
\end{equation}

We also introduce the following
\begin{equation}
\rho_n=\frac{1}{|\ttt_{n-1}|}\sum_{k\in\ttt_{n-1,p}}\vep_{2k}\vep_{2k+1}.
\end{equation}

In order to establish the MDP results of our estimators, we shall make use of
a martingale approach.
For all $n\ge p,$ denote
\begin{equation}\label{defM_n}
M_n=\sum_{k\in
\ttt_{n-1,p-1}}\begin{pmatrix}
\vep_{2k}\\
\vep_{2k}\xx_k\\
\vep_{2k+1}\\
\vep_{2k+1}\xx_k
\end{pmatrix} \in \rr^{2(p+1)}.
\end{equation}

We can clearly rewrite (\ref{thest}) as
\begin{equation}\label{thetaM_n}\hat\theta_n-\theta=\Sigma_{n-1}^{-1}M_n.\end{equation}

We know from Bercu et al. \cite{BSG09} that $(M_n)$ is  a square
integrable martingale adapted to the filtration $\ff=(\FF_n)$. Its
increasing process is given for all $n\ge p$ by
\[
\langle M\rangle_n=\Gamma\otimes S_{n-1}
\]
where $S_n$ is given in (\ref{defS_n}) and $\Gamma$ is given in (\ref{Gamma}).
\bigskip

We recall that for a sequence of random variables $(Z_{n})_{n}$ on
$\dR^{d \times p}$, we say that $(Z_{n})_{n}$ converges
$(b_{n}^2)-$superexponentially fast in probability to some random
variable $Z$ if, for all $\delta > 0$,
\begin{equation*}
\limsup_{n \rightarrow \infty} \frac{1}{b_{n}^2} \log \dP\Big( \left\Vert Z_{n} - Z \right\Vert > \delta \Big) = -\infty.
\end{equation*}
This \textnormal{exponential convergence} with speed $b_{n}^2$ will be shortened as
\begin{equation*}
Z_{n} \superexpp Z.
\end{equation*}
We follow Dembo and Zeitouni \cite{DemZei98} for the language of the
large deviations, throughout this paper. Before going further, let
us recall the definition of a MDP: let $(b_n)$ an increasing
sequence of positive real numbers such that
\begin{equation}\label{scaleMDP}b_n\longrightarrow\infty\qquad {\rm and}\qquad \frac{b_n}{\sqrt n}\longrightarrow 0.\end{equation}
We say that a sequence of centered random variables $(M_{n})_{n}$ with
topological state space $(S,{\mathcal S})$ satisfies a MDP
with speed $b_n^2$ and rate function $I:
S\rightarrow \mathbb{R}_{+}^{*}$ if for each $A\in {\mathcal S}$,

$$-\inf\limits_{x\in A^{o}}I(x)\leq
\liminf\limits_{n\rightarrow\infty}\frac{1}{b^2_{n}}\log\mathbb{P}\left(\frac{\sqrt{n}}{b_n}M_n\in
A\right)\leq
\limsup\limits_{n\rightarrow\infty}\frac{1}{b^2_{n}}\log\mathbb{P}\left(\frac{\sqrt{n}}{b_n}M_n\in
A\right)\leq -\inf\limits_{x\in \overline{A}}I(x),$$
here $A^{o}$ and  $\overline{A}$ denote the interior and closure of $A$ respectively.\\

\bigskip
Before the presentation of the main results, let us fix some more
notation. Let $\displaystyle \overline{a}=\frac{a_0+b_0}{2},\quad
\overline{a^2}=\frac{a^2_0+b^2_0}{2},\quad
\overline{A}=\frac{A+B}{2}$ and $e_1=(1,0,\cdots,0)^t\in \rr^p.$ We
denote

\begin{equation}\label{defL_12}
\Xi=\overline{a}(I_p-\overline{A})^{-1}e_1,
\end{equation}
and $\Lambda$   the unique solution of the equation
\begin{equation}\label{defL_22}
\Lambda=T+\frac 12(A\Lambda A^t+B\Lambda B^t)
\end{equation}
where
\begin{equation}\label{defT}
T=\left(\sigma^2+\overline{a^2}\right)e_1e_1^t+\frac
12\left(a_0\left(A\Xi e_1^t+e_1\Xi^tA^t\right)+ b_0\left(B\Xi
e_1^t+e_1\Xi^tB^t\right)\right),
\end{equation}
We also introduce the following matrix $L$ and $\Sigma$  given by
\begin{equation}\label{defL}L=\begin{pmatrix}
1  & \Xi\\
\Xi & \Lambda
\end{pmatrix}\quad \text{and} \quad
\Sigma = I_{2}\otimes L.
\end{equation}

\brmk In the special case $p=1$, we have $\displaystyle
\Xi=\frac{\overline{a}}{1-\overline{ b}}$, and
$\displaystyle\Lambda=\frac{\overline{a^2}+\sigma^2+2\Xi\overline{ab}}{1-\overline{b^2}}$,
where  $\displaystyle \overline{a b}=\frac{a_0a_1+b_0b_1}{2},
\overline{b}=\frac{a_1+b_1}{2},
\overline{b^2}=\frac{a^2_1+b^2_1}{2}.$
 \nrmk

\bigskip

\section{Main results}
Let us present now the main results of this paper.  In the following
theorem, we will give the deviation inequalities of the estimator of
the parameters, 1 useful for non asymptotic statistical
studies.
\bthm\label{thm:deviation_ineq_theta}$\,$
\begin{enumerate}
\item [\rm{\bf (i)}] In the case 1, we have for all $\delta>0$ and
for all $b>0$ such that $b < \|\Sigma\|/(1+ \delta)$
\begin{equation}\label{dev_ineq_theta1}
\pp\left(\|\hat{\theta}_{n} - \theta\| > \delta\right) \leq
\begin{cases}
c_{1}\exp\left(-\frac{c_{2}(\delta b)^{2}}{c_{3} + (\delta b)}
\frac{2^{n}}{(n-1)^{2}}\right) \hspace{1.75cm} \text{if $\beta <
\frac{1}{2}$} \\ \\ c_{1}(n-1) \exp\left(\frac{-c_{2} (\delta
b)^{2}}{c_{3} + (\delta b)} \frac{2^{n}}{(n-1)^{2}}\right)
\hspace{0.5cm} \text{if $\beta = \frac{1}{2}$} \\ \\ c_{1}(n-1)
\exp\left(\frac{-c_{2} (\delta b)^{2}}{c_{3} + (\delta
b)}\frac{1}{(n-1)\beta^{n}}\right) \hspace{0.35cm} \text{if $\beta >
\frac{1}{2}$},
\end{cases}
\end{equation}
where the constants $c_{1}$, $c_{2}$ and $c_{3}$ depend on
$\sigma^2,$ $\beta,$ $\gamma$ and $\phi$ and are such that $c_{1},
c_{2}>0$, $c_{3}\geq 0.$

\item [\rm{\bf (ii)}] In the case 2, we have for all $\delta>0$ and
for all $b>0$ such that $b < \|\Sigma\|/(1+ \delta)$
\begin{equation}\label{dev_ineq_theta2}
\pp\left(\|\hat{\theta}_{n} - \theta\| > \delta\right) \leq
\begin{cases}
c_{1}\exp\left(-\frac{c_{2}(\delta b)^{2}}{c_{3} + c_{4}(\delta
b)}\frac{2^{n}}{(n-1)^{2}}\right) \hspace{0.6cm} \text{if $\beta <
\frac{\sqrt{2}}{2}$} \\ \\ c_{1} \exp\left(-\frac{c_{2}(\delta
b)^{2}}{c_{3} + c_{4}(\delta b)}\frac{2^{n}}{(n-1)^{3}}\right)
\hspace{0.5cm} \text{if $\beta = \frac{\sqrt{2}}{2}$} \\ \\
c_{1} \exp\left(-\frac{c_{2}(\delta b)^{2}}{c_{3} + c_{4}(\delta
b)}\frac{1}{(n-1)^{2}\beta^{2n}}\right) \hspace{0.35cm} \text{if
$\beta > \frac{\sqrt{2}}{2}$},
\end{cases}
\end{equation}
where the constants $c_{1}$, $c_{2}$, $c_{3}$, and  $c_{4}$ depend
on $\sigma^2,$ $\beta,$ $\gamma$ and $\phi$ and are such that
$c_{1}, c_{2} >0$, $c_{3} , c_{4} \geq 0$, $(c_{3},c_{4})\neq
(0,0).$
\end{enumerate}
\nthm

\brmk
One can notice that the estimate (\ref{dev_ineq_theta2}) is
stronger than estimate (\ref{dev_ineq_theta1}). This is due to the
fact that the integrability condition in case 2 is stronger than
integrability condition in case 1.
\nrmk

\brmk
The upper bounds in previous theorem holds for arbitrary
$n\geq p-1$ (not a limit relation, unlike the below results), hence
they are much more practical (in non asymptotic statistics).
\nrmk

In the next result, we will present the MDP of  the estimator
$\hat\theta_n$.

\bthm\label{thm:mdp_theta_n} In the case 1 or in the case 2, the sequence $\displaystyle\left(\sqrt{|\ttt_{n-1}|} (\hat\theta_n-\theta)/b_{|\ttt_{n-1}|}\right)_{n\ge 1}$
satisfies the MDP on $\rr^{2(p+1)}$ with  speed
$b^2_{|\ttt_{n-1}|}$ and rate function
\begin{equation}\label{Itheta}
I_{\theta}(x)=\sup_{\lambda\in \rr^{2(p+1)}}\{\lambda^t x-\lambda (\Gamma\otimes L^{-1})\lambda^t\}=\frac 12 x^t  (\Gamma\otimes L^{-1})^{-1} x,
\end{equation}
where $L$ and $\Gamma$ are given in (\ref{defL}) and (\ref{Gamma}) respectively.
\nthm

\brmk  Similar results about deviation inequalities and MDP,  are already obtained in \cite{BDG11},
in a restrictive case of bounded or gaussian noise and when $p=1$, but results therein hold for general Markov models also.
\nrmk

Let us consider now the estimation of the parameter in the noise process.

\bthm\label{mdp_sigma_n_rho_n} Let $(b_n)$ an increasing sequence of positive real numbers such that
$$b_n\longrightarrow\infty\qquad {\rm and}\qquad \frac{b_n}{\sqrt n}\longrightarrow 0.$$
In the case 1 or in the case 2,
\begin{enumerate}
\item[(1)] the sequence  $\displaystyle\left(\sqrt{|\ttt_{n-1}|} (\sigma^2_n-\sigma^2)/b_{|\ttt_{n-1}|}\right)_{n\ge 1}$
satisfies the MDP on $\rr$ with  speed
$b^2_{|\ttt_{n-1}|}$ and rate function
\begin{equation}\label{Isigma}
I_{\sigma^2}(x)=\frac{x^2}{\tau^4-2\sigma^4+\nu^2}.
\end{equation}

\item[(2)] the sequence  $\displaystyle\left(\sqrt{|\ttt_{n-1}|}(\rho_n-\rho)/b_{|\ttt_{n-1}|}\right)_{n\ge 1}$
satisfies the MDP on $\rr$ with  speed
$b^2_{|\ttt_{n-1}|}$ and rate function
\begin{equation}\label{Irho}
I_{\rho}(x)=\frac{x^2}{2(\nu^2-\rho^2)}.
\end{equation}
\end{enumerate}
\nthm

\brmk Note that in this case the MDP holds for all the scale $(b_n)$ verifying (\ref{scaleMDP}) without other restriction. \nrmk

\brmk It will be more interesting to prove the MDP for
$\displaystyle\left(\sqrt{|\ttt_{n-1}|}
(\hat\sigma^2_n-\sigma^2)/b_{|\ttt_{n-1}|}\right)_{n\ge 1}$, which
will be the case if one proves for example that
$\displaystyle\left(\sqrt{|\ttt_{n-1}|}
(\hat\sigma^2_n-\sigma^2)/b_{|\ttt_{n-1}|}\right)_{n\ge 1}$ and
$\displaystyle\left(\sqrt{|\ttt_{n-1}|}
(\sigma^2_n-\sigma^2)/b_{|\ttt_{n-1}|}\right)_{n\ge 1}$  are
exponentially equivalent in the sense of the MDP. This is described
by the following convergence
$$
\frac{\sqrt{|\ttt_{n-1}|}}{b_{|\ttt_{n-1}|}}(\hat{\sigma}_{n}^{2} -
\sigma_{n}^{2}) \superexp 0.$$
The proof is very technical and very restrictive for the scale of the deviation. Actually we are only able to prove that
$$
\hat{\sigma}_{n}^{2} -\sigma_{n}^{2}\superexp 0,$$ this
superexponential convergence will be proved in Theorem
\ref{thm:conv_expo_sigma_n_rho_n}. \nrmk

In the following theorem we will state the superexponential convergence.

\bthm\label{thm:conv_expo_sigma_n_rho_n}  In the case 1 or in the
case 2,  we have
\begin{equation*}
\hat{\sigma}_{n}^{2}\superexp \sigma^{2}.
\end{equation*}
In the case 1, instead of {\bf (G2)}, if we assume that

{\bf (G2')}  one can find $\gamma'>0$ such that for all $n\geq p-1,$
for all $k,l\in \g_{n+1}$ with $[\frac{k}{2}] = [\frac{l}{2}]$ and
for all $t\in ]-c,c[$ for some $c>0,$
\[
\ee\left[\exp t\left(\vep_{k}\vep_{l} -
\rho\right)\right]\leq \exp\left(\frac{\gamma'
t^{2}}{2}\right),
\]

and in the case 2, instead of {\bf (Ea)}, if we assume that

{\bf (E2')}  one can find $\gamma'>0$ such that for all $n\geq p-1,$
for all $k,l\in \g_{n+1}$ with $[\frac{k}{2}] = [\frac{l}{2}]$ and
for all $t\in \rr$
\[
\ee\left[\exp t\left(\vep_{k}\vep_{l} -
\rho\right)/\FF_{n}\right]\leq \exp\left(\frac{\gamma'
t^{2}}{2}\right),\qquad a.s.
\]

then in the case 1 or in the case 2, we have
\begin{equation*}
 \hat{\rho}_{n} \superexp \rho.
\end{equation*}
\nthm

\bigskip

Before going to the proofs, let us gather here for the convenience
of the readers two Theorems useful to establish MDP of the
martingales and used intensively in this paper. From this two
theorems, we will be able to give a strategy for the proof.

Let $M=(M_{n}, \mathcal{H}_{n}, n\geq 0)$ be a centered square
integrable martingale defined on a probability space $(\Omega,
\mathcal{H}, \mathbb{P})$ and $(\langle M\rangle_{n})$ its bracket. Let $(b_n)$ an increasing sequence
of real numbers satisfying (\ref{scaleMDP}). Let us enunciate the following which corresponds to the
unidimensional case of Theorem 1 in \cite{Dj02}.

\begin{prop}\label{mdpdj} Let $c(n):=\frac{\sqrt{n}}{b_{n}}$ is non-decreasing, and define the
reciprocal function $c^{-1}(t)$ by

\[c^{-1}(t):=\inf\{n\in \mathbb{N}: c(n)\geq t\}.\]

Under the following conditions:

\begin{enumerate}
\item [\rm(\textbf{D1})] there exists $Q\in \mathbb{R}_{+}^{*}$ such that
$\displaystyle \frac{\langle M\rangle_{n}}{n}\superexpp  Q$;
\item [\rm(\textbf{D2})]  $\displaystyle \limsup\limits_{n\rightarrow
+\infty}\frac{n}{b_{n}^{2}}\log\left(n \quad \underset{1\leq k\leq
c^{-1}(\sqrt{n+1}b_{n+1})}{\rm ess\,sup}
\mathbb{P}(|M_{k}-M_{k-1}|>b_{n}\sqrt{n}/\mathcal{H}_{k-1})\right)=-\infty;$
\item [\rm(\textbf{D3})] for all $ a>0$ $\displaystyle \frac{1}{n}\sum\limits_{k=1}^{n}\mathbb{E}\left(|M_{k}-M_{k-1}|^{2}
\mathbf{1}_{\{|M_{k}-M_{k-1}|\geq a\frac{\sqrt
{n}}{b_{n}}\}}/\mathcal{H}_{k-1} \right)\superexpp  0;$
\end{enumerate}

$(M_{n}/b_{n}\sqrt{n})_{n\in \mathbb{N}}$ satisfies the MDP in
$\mathbb{R}$ with the speed $b_{n}^{2}$ and the rate function
$\displaystyle I(x) =\frac{x^{2}}{2Q}.$
\end{prop}

Let us introduce a simplified version of Puhalskii's result
\cite{Puha97} applied to a sequence of martingale differences.

\begin{thm}\label{mdppuh} Let $(m_{j}^{n})_{1 \leq j \leq n}$ be a triangular
array of martingale differences with values in $\dR^{d}$, with
respect to the filtration $(\mathcal{H}_{n})_{n \geq 1}$.  Under the
following conditions
\begin{enumerate}
\item[\rm(\textbf{P1})] there exists a symmetric positive semi-definite matrix $Q$ such that
$$\frac{1}{n} \sum_{k=1}^{n} \dE\Big[ m_{k}^{n} (m_{k}^{n})^{\prime} \big\vert \mathcal{H}_{k-1} \Big] \superexpp Q,$$
\item[\rm(\textbf{P2})] there exists a constant $c > 0$ such that, for each $1 \leq k \leq n$,
$\vert m_{k}^{n} \vert \leq c \frac{\sqrt{n}}{b_{n}} \hspace{0.5cm}
\textnormal{a.s.},$
\item[\rm(\textbf{P3})] for all $a > 0$, we have the exponential Lindeberg's condition
$$\frac{1}{n} \sum_{k=1}^n \dE\Big[\vert m_{k}^{n} \vert^2 \mathrm{I}_{\left\{ \vert m_{k}^{n} \vert \geq a \frac{\sqrt{n}}{b_{n}} \right\}} \big\vert \mathcal{H}_{k-1} \Big] \superexpp 0.$$
\end{enumerate}
$(\sum_{k=1}^{n} m_{k}^{n}/(b_{n} \sqrt{n}))_{n \geq 1}$ satisfies
an MDP on $\dR^{d}$ with speed $b_{n}^2$ and rate function
\begin{equation*}
\Lambda^*(v) = \sup_{\lambda \in \dR^{d}} \left( \lambda^{\prime} v
- \frac{1}{2}\lambda^{\prime} Q \lambda\right).
\end{equation*}
In particular, if $Q$ is invertible, $\Lambda^*(v) = \frac{1}{2}
v^{\prime} Q^{-1} v.$
\end{thm}
As the reader can imagine naturally now, the strategy of the proof
of the MDP  consist on the following steps :
\begin{itemize}
\item the superexponential convergence of the quadratic variation of the martingale $(M_n)$.
This step is very crucial and the key for the rest of the paper. It
will be realized by means of powerful exponential inequalities. This
allows us to obtain the deviation inequalities for the estimator of
the parameters,
\item introduce a truncated martingale which satisfies the MDP, thanks to a classical
theorems \ref{mdppuh},
\item the truncated martingale is an exponentially good approximation of $(M_n)$, in the sense of the moderate deviation.
\end{itemize}

\section{Superexponential convergence of the quadratic variation of the martingale }\label{appendixProp}

At first, it is necessary to establish the superexponential
convergence of  the quadratic variation of the martingale $(M_n)$,
properly normalized in order to prove the MDP, of the estimators.
Its proof is very technical, but crucial for the rest of the paper.
This section contains also some deviation inequalities for some
quantities needed in the proof later.

\bprop\label{prop:convergence_crochet} In the case 1 or case 2, we have
\begin{equation}\label{convergence_crochet}
\frac{S_n}{|\ttt_n|}\superexpn L,
\end{equation}
where $S_n$ is given in (\ref{defS_n}) and $L$ is given in (\ref{defL}).
 \nprop
For the proof we focus in the case 2. The Proposition \ref{convergence_crochet} will follows from Proposition \ref{convexpocrochet1} and Proposition \ref{convexpocrochet2} below, where we assume that the sequence  $(b_n)$ satisfies the
condition {\bf (V2)}. Proposition \ref{convexpocrochet3} gives some ideas of the proof in the case 1.

\brmk Using \cite{DjGuWu06}, we infer from {\bf (Ea)} that
\begin{enumerate}
\item[{\bf (N2)}]  one can find $\gamma>0$ such
that for all $n\geq p-1,$ for all $k\in \g_{n+1}$ and for all $t\in
\rr$
\[
\ee\left[\exp t\left(\vep_{k}^{2} - \sigma^{2}\right)/\FF_{n}\right]\leq
\exp\left(\frac{\gamma t^{2}}{2}\right)\qquad a.s.
\]
\end{enumerate}
\nrmk

\begin{prop}\label{convexpocrochet1}
Assume that hypothesis {\bf (N2)} and {\bf (Xa)} are satisfied. Then
we have
\[
\frac{1}{|\ttt_{n}|}\sum\limits_{k\in\ttt_{n,p}}\xx_{k} \superexpn \Xi,
\]
where $\Xi$ is given in (\ref{defL_12}).
\end{prop}

\begin{proof}
Let
\[
H_n=\sum_{k\in \ttt_{n,p-1}}\xx_k \qquad {\rm
and}\qquad P_n=\sum_{k\in \ttt_{n,p}}\epsilon_k.
\]
From Bercu et al. \cite{BSG09}, we have
\begin{equation}\label{H_n}
\frac{H_{n}}{2^{n+1}} =
\sum\limits_{k=p-1}^{n} (\overline{A})^{n-k}\frac{H_{p-1}}{2^{k}} +
\sum\limits_{k=p}^{n}
\overline{a}(\overline{A})^{n-k}\left(\frac{2^{k}-2^{p-1}}{2^{k}}\right)e_{1}
+ \sum\limits_{k=p}^{n}
\frac{P_{k}}{2^{k+1}}(\overline{A})^{n-k}e_{1}.
\end{equation}
Since the second term in the right hand side of this equality is deterministic, this proposition will be proved if we show that
\begin{equation}\label{conv_sum_H_p1}
\sum\limits_{k=p-1}^{n} \frac{(\overline{A})^{n-k}}{2^{k}}H_{p-1}
\superexpn 0, \qquad \sum\limits_{k=p}^{n}
\frac{P_{k}}{2^{k+1}}(\overline{A})^{n-k}e_{1} \superexpn 0,
\end{equation}
which follows by performing as in the proof of Proposition
\ref{convexpocrochet2} (see the proof of Proposition
\ref{convexpocrochet2} for more details).
\end{proof}
\begin{prop}\label{convexpocrochet2}
Assume that hypothesis {\bf (N2)} and {\bf (Xa)} are satisfied. Then
we have
\[
\frac{1}{|\ttt_{n}|}\sum\limits_{k\in\ttt_{n,p}}\xx_{k}\xx_{k}^{t}
\superexpn \Lambda,
\]
where $\Lambda$ is given in (\ref{defL_22}).
\end{prop}

\begin{proof}
Let
\begin{equation}\label{KL}
K_{n}=\sum\limits_{k\in\ttt_{n,p-1}}\xx_{k}\xx_{k}^{t}\qquad
\text{and} \qquad
L_{n}=\sum\limits_{k\in\ttt_{n,p}}\vep_{k}^{2}.\end{equation} Then
from (\ref{bar_p2}), and after straightforward calculations (see
\cite{BSG09} for more details), we get  that
\[
\frac{K_{n}}{2^{n+1}}=\frac{1}{2^{n-p+1}}\sum\limits_{C\in\{A;B\}^{n-p+1}}C\frac{K_{p-1}}{2^{p}}C^{t}
+
\sum\limits_{k=0}^{n-p}\frac{1}{2^{k}}\sum\limits_{C\in\{A;B\}^{k}}CT_{n-k}C^{t},
\]
where  the notation $\{A;B\}^k$ means the set of all products of $A$
and $B$  with exactly $k$ terms. The cardinality of $\{A;B\}^k$  is
obviously $2^k$, and
$$T_k =\frac{L_{k}}{2^{k+1}}e_1e_1^t+\overline{a^2}\left(\frac{2^{k}-2^{p-1}}{2^{k}}\right)e_1e_1^t+I_k^{(1)}+ I_k^{(2)}+\frac{1}{2^{k+1}}U_k$$

with $\overline{a^2}=(a_0^2+b_0^2)/2$ and
\begin{equation}\label{I_n,1}I_{k}^{(1)}=\frac 12
\left(a_0\left(A\frac{H_{k-1}}{2^{k}}e_1^t+e_1\frac{H_{k-1}}{2^{k}}A^t\right)+b_0\left(B\frac{H_{k-1}}
{2^{k}}e_1^t+e_1\frac{H_{k-1}}{2^{k}}B^t\right)\right),
\end{equation}

\begin{equation}\label{I_n,2}
I_{k}^{(2)}=\left(\frac{1}{2^{k}}\sum_{l\in
\ttt_{k-1,p-1}}(a_0\vep_{2l}+b_0\vep_{2l+1})\right)e_1e_1^t,\end{equation}

\begin{equation}\label{U_k}U_{k}=\sum\limits_{l\in\ttt_{k-1,p-1}}\vep_{2l}\Big(A\xx_le_1^t+e_1\xx_l^tA^t\Big)+\vep_{2l+1}\Big(B\xx_le_1^t+e_1\xx_l^tB^t\Big).\end{equation}

\medskip

Then proposition will follow if we prove Lemmas \ref{lconv1},  \ref{lconv2}, \ref{lconv3}, \ref{lconv4} and  \ref{lconv5}.
\begin{lem}\label{lconv1} Assume that hypothesis  {\bf (Xa)} is satisfied. Then
we have
\begin{equation}\label{conv1}
\frac{1}{2^{n-p+1}}\sum\limits_{C\in\{A;B\}^{n-p+1}}C\frac{K_{p-1}}{2^{p}}C^{t}\superexpn
0,
\end{equation}
where $K_p$ is given in (\ref{KL}).
\end{lem}

\begin{proof} We get easily
\[
\left\|\frac{1}{2^{n-p+1}}\sum\limits_{C\in\{A;B\}^{n-p+1}}C\frac{K_{p-1}}{2^{p}}C^{t}\right\|\leq
c\beta^{2n}\overline{X}_{1}^{2},
\]
where $\beta$ is given in (\ref{beta}), $\overline{X}_{1}$ is introduced in
{\bf(Xa)} and $c$ is a positive constant which depends on $p$. Next,
Chernoff inequality and hypothesis {\bf (X2)} lead us easily to
(\ref{conv1}).
\end{proof}

\begin{lem} \label{lconv2} Assume that hypothesis {\bf (N2)} and {\bf (Xa)} are satisfied. Then
we have

\begin{equation}\label{conv2}
\sum\limits_{k=0}^{n-p}\frac{1}{2^{k}}\sum\limits_{C\in\{A;B\}^{k}}C\frac{L_{n-k}}
{2^{n-k}}e_{1}e_{1}^{t}C^{t}\superexpn
\sigma^{2}e_{1}e_{1}^{t},
\end{equation}
where $L_k$ is given in the second part of (\ref{KL}).
\end{lem}

\begin{proof} First, since we have for all $k\geq p$ the following
  decomposition on odd and even part
$$\sum\limits_{i\in\ttt_{k,p}}(\vep_{i}^{2}-\sigma^{2})=\sum\limits_{i\in\ttt_{k-1,p-1}}(\vep_{2i}^{2}-\sigma^{2})+(\vep_{2i+1}^{2}-\sigma^{2}),$$
we obtain for all $\delta>0$ that
\begin{equation*}
\pp\left(\frac{1}{|\ttt_{k}|+1}\sum\limits_{i\in\ttt_{k,p}}(\vep_{i}^{2}-\sigma^{2})>\delta\right)\leq\sum_{\eta=0}^{1}
\pp\left(\frac{1}{|\ttt_{k}|+1}\sum\limits_{i\in\ttt_{k-1,p-1}}(\vep_{2i+\eta}^{2}-\sigma^{2})>\frac{\delta}{2}\right).
\end{equation*}
We will treat only the case $\eta=0$. Chernoff inequality gives us for all $\lambda>0$
\[
\pp\left(\frac{1}{|\ttt_{k}|+1}\sum\limits_{i\in\ttt_{k-1,p-1}}(\vep_{2i}^{2}-\sigma^{2})>\frac{\delta}{2}\right)
\leq
\exp\left(-\lambda\frac{\delta}{2}2^{k+1}\right)\ee\left[\exp\left(\lambda\sum\limits_{i\in\ttt_{k-1,p-1}}(\vep_{2i}^{2}
-\sigma^{2})\right)\right].
\]
We obtain from hypothesis {\bf (N2)}, after conditioning by
$\FF_{k-1}$
\[
\ee\left[\exp\left(\lambda\sum\limits_{i\in\ttt_{k-1,p-1}}(\vep_{2i}^{2}
-\sigma^{2})\right)\right]\leq
\exp\left(\lambda^{2}\gamma|\g_{k-1}|\right)\ee\left[\exp\left(\lambda\sum\limits_{i\in\ttt_{k-2,p-1}}(\vep_{2i}^{2}
-\sigma^{2})\right)\right].
\]
Iterating this, we deduce that
\[
\ee\left[\exp\left(\lambda\sum\limits_{i\in\ttt_{k-1,p-1}}(\vep_{2i}^{2}
-\sigma^{2})\right)\right]\leq
\exp\left(\gamma\lambda^{2}\sum\limits_{l=p-1}^{k-1}|\g_{l}|\right)\leq
\exp\left(\gamma\lambda^{2}2^{k+1}\right).
\]
Next, optimizing on $\lambda$, we get
\[
\pp\left(\frac{1}{|\ttt_{k}|+1}\sum\limits_{i\in\ttt_{k-1,p-1}}(\vep_{2i}^{2}-\sigma^{2})>\frac{\delta}{2}\right)
\leq \exp\left(-c\delta^{2}|\ttt_{k}|\right)
\]
for some positive constant $c$ which depends on $\gamma$. Applying the foregoing to the random variables
$-(\vep_{i}^{2}-\sigma^{2})$, we obtain
\begin{equation}\label{devineqeps2}
\pp\left(\frac{1}{|\ttt_{k}|+1}\left|\sum\limits_{i\in\ttt_{k,p}}(\vep_{i}^{2}-\sigma^{2})\right|>\delta\right)\leq
4\exp\left(-c\delta^{2}|\ttt_{k}|\right).
\end{equation}
Next, from the following inequalities
\begin{eqnarray*}
\left\|\sum\limits_{k=0}^{n-p}\frac{1}{2^{k}}\sum\limits_{C\in\{A;B\}^{k}}C\frac{L_{n-k}-\sigma^{2}}{2^{n-k}}
e_{1}e_{1}^{t}C^{t}\right\|& \leq&
\sum\limits_{k=0}^{n-p}\frac{1}{2^{k}}\sum\limits_{C\in\{A;B\}^{k}}\frac{|L_{n-k}-\sigma^{2}|}{2^{n-k}}
\left\|Ce_{1}e_{1}^{t}C^{t}\right\|\\ & \leq&
\sum\limits_{k=p}^{n}\beta^{2(n-k)}\frac{|L_{k}-\sigma^{2}|}{|\ttt_{k}|+1}
\end{eqnarray*}

and from (\ref{devineqeps2}) applied with $\delta/((n-p+1)\beta^{2(n-k)})$ instead of $\delta$, we get
\begin{align}
\notag\pp\left(\left\|\sum\limits_{k=0}^{n-p}\frac{1}{2^{k}}\sum\limits_{C\in\{A;B\}^{k}}C\frac{L_{n-k}-\sigma^{2}}{2^{n-k}}
e_{1}e_{1}^{t}C^{t}\right\|>\delta\right)&\leq
\pp\left(\sum\limits_{k=p}^{n}\beta^{2(n-k)}\frac{|L_{k}-\sigma^{2}|}{|\ttt_{k}|+1}>\delta\right)
\\ \notag &\leq
\sum\limits_{k=p}^{n}\pp\left(\frac{|L_{k}-\sigma^{2}|}{|\ttt_{k}|+1}>\frac{\delta}{(n-p+1)\beta^{2(n-k)}}\right)
\\ \notag &\leq
c_{1}\sum\limits_{k=p}^{n}\exp\left(-c_{2}\delta^{2}\frac{(2\beta^{4})^{k+1}}{n^{2}\beta^{4n}}\right).
\end{align}
Now, following the same lines as in the proof of
(\ref{conv5}) we obtain
\begin{equation}\label{ineq_L}
\pp\left(\left\|\sum\limits_{k=0}^{n-p}\frac{1}{2^{k}}\sum\limits_{C\in\{A;B\}^{k}}C\frac{L_{n-k}-\sigma^{2}}{2^{n-k}}
e_{1}e_{1}^{t}C^{t}\right\|>\delta\right)\leq\begin{cases}
c_{1}\exp\left(-c_{2}\delta^{2}\frac{2^{n+1}}{n^{2}}\right)
\hspace{0.25cm} \text{if $\beta^{4}<\frac{1}{2}$}, \\
c_{1}n\exp\left(-c_{2}\delta^{2}\frac{2^{n+1}}{n^{2}}\right)
\hspace{0.25cm} \text{if $\beta^{4}=\frac{1}{2}$}, \\
c_{1}\exp\left(-c_{2}\delta^{2}\frac{1}{n^{2}\beta^{4n}}\right)
\hspace{0.25cm} \text{if $\beta^{4}>\frac{1}{2}$}, \end{cases}
\end{equation}
for some positive constants $c_{1}$ and $c_{2}$. From (\ref{ineq_L}), we infer
that (\ref{conv2}) holds.

\end{proof}


\begin{lem} \label{lconv3} Assume that hypothesis {\bf (N2)} is satisfied. Then
we have

\begin{equation} \label{conv3}
\sum\limits_{k=0}^{n-p}\frac{1}{2^{k}}\sum\limits_{C\in\{A;B\}^{k}}CI_{n-k}^{(2)}
C^{t}\superexpn 0,
\end{equation}
where $I_k^{(2)}$ is given in (\ref{I_n,2}).
\end{lem}
\begin{proof} This proof follows the same
lines as that of (\ref{conv2}) and uses hypothesis {\bf (N1)}
instead of {\bf (N2)}.
\end{proof}


\begin{lem} \label{lconv4} Assume that hypothesis {\bf (N2)} and {\bf (Xa)} are satisfied. Then
we have
\begin{equation}\label{conv4}
\sum\limits_{k=0}^{n-p}\frac{1}{2^{k}}\sum\limits_{C\in\{A;B\}^{k}}CI_{n-k}^{(1)}C^{t}
\superexpn \Lambda', \quad \text{where} \hspace{0.2cm}
\Lambda'=T-(\sigma^{2}+\overline{a^{2}})e_{1}e_{1}^{t},
\end{equation}
where $T$ is given (\ref{defT}) and $I_k^{(1)}$ is given in (\ref{I_n,1}).
\end{lem}

\begin{proof} Since in the definition of $I_n^{(1)}$ given by (\ref{I_n,1}),
  there are four terms, we will focus only on the first term
$$\frac{a_0}{2}A\frac{H_{k-1}}{2^{k}}e_1^t,$$
the other terms will be treated in the same way. Using (\ref{H_n}),
we obtain the following decomposition:

$$\frac{a_0}{2}\sum\limits_{k=0}^{n-p}\frac{1}{2^{k}}\sum\limits_{C\in\{A;B\}^{k}}CA
\frac{H_{n-k-1}}{2^{n-k}}e_{1}^{t}C^{t}=T^{(1)}_{n}+T^{(2)}_{n}+T^{(3)}_{n}$$
where
$$T^{(1)}_{n} = \frac{a_0}{2}\sum\limits_{k=0}^{n-p}\frac{1}{2^k}\sum\limits_{C\in\{A;B\}^{k}}
CA\left\{\overline{A}^{n-k-p}\frac{H_{p-1}}{2^p} +
\sum\limits_{l=p}^{n-k-1}\overline{A}^{n-k-l-1}\frac{H_{p-1}}{2^{l+1}}\right\}e_1^{t}C^{t},$$

$$T^{(2)}_{n}=\frac{a_0}{2}\sum\limits_{k=0}^{n-p} \frac{1}{2^{k}}\sum\limits_{C\in\{A;B\}^{k}}
CA\left\{\sum\limits_{l=p}^{n-k-1}
\overline{A}^{n-k-l-1}\overline{a}\left(
\frac{2^{l}-2^{p-1}}{2^{l}}\right)e_{1}e_{1}^{t}\right\}C^{t},$$ and
$$T^{(3)}_{n}=\frac{a_0}{2}\sum\limits_{k=0}^{n-p}\frac{1}{2^{k}}\sum\limits_{C\in\{A;B\}^{k}}
CA\sum\limits_{l=p}^{n-k-1}
\overline{A}^{n-k-l-1}\frac{P_{l}}{2^{l+1}}e_{1}e_{1}^{t}C^{t}.$$ On
the one hand we have
\[
\|T^{(3)}_{n}\|\leq
c\sum\limits_{k=p}^{n}\beta^{n-k}\frac{|P_{k}|}{2^{k+1}}
\]
where $c$ is a positive constant such that
$c>|a_{0}|\frac{1-\beta^{n-l}}{1-\beta}$ for all $n\geq l$, so that
\[
\pp\Big(\|T^{(3)}_{n}\|>\delta\Big) \leq
\sum\limits_{k=p}^{n}\pp\left(\frac{|P_{k}|}{|\ttt_{k}|+1}>\frac{2\delta}{cn\beta^{n-k}}\right).
\]
We deduce again from hypothesis {\bf (N1)} and in the same way we
have obtained (\ref{devineqeps2}) that
\[
\pp\left(\frac{P_{k}}{|\ttt_{k}|+1}>\frac{2\delta}{cn\beta^{n-k}}\right)\leq
\exp\left(-c_{1}\delta^{2}\frac{(2\beta^{2})^{k+1}}{n^{2}\beta^{2n}}\right)
\hspace{0.25cm} \forall k\geq p,
\]
for some positive constant $c_{1}.$ It then follows as in the proof
of (\ref{conv5}) that
\[
\pp\Big(\|T^{(3)}_{n}\|>\delta\Big)\leq\begin{cases}
\exp\left(-c_{1}\delta^{2}\frac{2^{n+1}}{n^{2}}\right)
\hspace{0.25cm} \text{if $\beta^{2}<\frac{1}{2}$}, \\
n\exp\left(-c_{1}\delta^{2}\frac{2^{n+1}}{n^{2}}\right)
\hspace{0.25cm} \text{if $\beta^{2}=\frac{1}{2}$}, \\
\exp\left(-c_{1}\delta^{2}\frac{1}{n^{2}\beta^{2n}}\right)
\hspace{0.25cm} \text{if $\beta^{2}>\frac{1}{2}$},\end{cases}
\]
so that
\begin{equation}\label{conv4_1}
T^{(3)}_{n}\superexpn 0.
\end{equation}
On the other hand, we have after studious calculations
\[
\|T^{(1)}_{n}\|\leq\begin{cases} c\frac{\overline{X}_{1}}{2^{n+1}}
\hspace{0.25cm} \text{if
$\beta<\frac{1}{2}$}, \\
c\frac{\overline{X}_{1}}{\sqrt{|\ttt_{n}|+1}} \hspace{0.25cm}
\text{if $\beta=\frac{1}{2}$}, \\ c\beta^{n}\overline{X}_{1}
\hspace{0.25cm} \text{if $\beta>\frac{1}{2}$},\end{cases}
\]
where $c$ is a positive constant which depends on $p$ and $|a_{0}|.$
Next, from hypothesis {\bf (X2)} and Chernoff inequality we
conclude that
\begin{equation}\label{conv4_2}
T^{(1)}_{n}\superexpn 0.
\end{equation}
Furthermore, since $(T^{(2)}_{n})$ is a deterministic sequence, we have
\begin{equation}\label{conv4_3}
T^{(2)}_{n} \superexpn \frac{1}{2}a_0A\Xi e_{1}^{t}.
\end{equation}
It then follows that
\[
\frac{a_0}{2}\sum\limits_{k=0}^{n-p}\frac{1}{2^{k}}\sum\limits_{C\in\{A;B\}^{k}}CA
\frac{H_{n-k-1}}{2^{n-k}}e_{1}^{t}C^{t} \superexpn
\frac{1}{2}a_{0}A\Xi e_{1}^{t}.
\]
Doing the same for the three other terms of $I_{k}^{(1)}$, we end the proof of Lemma (\ref{lconv4}).
\end{proof}

\begin{lem} \label{lconv5} Assume that hypothesis {\bf (N2)} and {\bf (Xa)} are satisfied. Then
we have

\begin{equation}\label{conv5}
\sum\limits_{k=0}^{n-p}\frac{1}{2^{k}}\sum\limits_{C\in\{A;B\}^{k}}C
\frac{U_{n-k}}{2^{n-k+1}}C^{t}\superexpn 0,
\end{equation}
where $U_k$ is given by (\ref{U_k}).
\end{lem}

\begin{proof} Let $\displaystyle V_{n}=\sum\limits_{k=2^{p-1}}^{n}\vep_{2k}X_{k}.$
Then $(V_{n})$ is a $\GG_{n}$-martingale and  its increasing process verifies
that
\begin{equation*}
\<V\>_{n}=\sigma^{2}\sum\limits_{k=2^{p-1}}^{n}X_{k}^{2} \leq
\sigma^{2}\sum\limits_{k=2^{p-1}}^{n}\|\xx_{k}\|^{2}\leq\sigma^{2}\sum\limits_{k\in\ttt_{r_{n},p-1}}\|\xx_{k}\|^{2}
\end{equation*}
From \cite{BSG09}, with $\alpha=\max(|a_0|,|b_0|)$, we have
\begin{equation}\label{decom_crochet_V}
\sum\limits_{k\in\ttt_{r_{n},p-1}}\|\xx_{k}\|^{2}\leq\frac{4}{1-\beta}P_{r_{n}}+\frac{4\alpha^{2}}{1-\beta}Q_{r_{n}}+2\overline{X}_{1}^{2}R_{r_{n}},\end{equation}
where
\[
P_{r_{n}}=\sum\limits_{k\in\ttt_{r_{n},p}}\sum\limits_{i=0}^{r_{k}-p}\beta^{i}\vep_{[\frac{k}{2^{i}}]}^{2},\quad
Q_{r_{n}}=\sum\limits_{k\in\ttt_{r_{n},p}}\sum\limits_{i=0}^{r_{k}-p}\beta^{i},\quad
R_{r_{n}}=\sum\limits_{k\in\ttt_{r_{n},p-1}}\beta^{2(r_{k}-p+1)}.
\]
For $\lambda>0$, we infer from hypothesis {\bf (N1)} that $(Y_{k})_{2^{p-1}\leq
k\leq n}$ given by
\[
Y_{n}=\exp\left(\lambda V_n-\frac{\lambda^{2}\phi}{2}
\sum\limits_{k=2^{p-1}}^{n}X_{k}^{2}\right),
\]
is a $\GG_{k}$-supermartingale and moreover $\ee\Big[Y_{2^{p-1}}\Big]\leq 1$.

For $B>0$ and $\delta>0$, we have
\begin{align}
\notag\pp\left(\frac{V_{n}}{2n}>\delta\right)&\leq
\pp\Big(\frac{\phi}{2n}\sum\limits_{i=2^{p-1}}^{n}X_{k}^{2}>B\ \Big)
+
\pp\left(Y_{n}>\exp\left(\lambda\delta-\frac{\lambda^{2}B}{2}\right)2n\right)\\
\notag &\leq
\pp\left(\frac{\phi}{2n}\sum\limits_{k=2^{p-1}}^{n}X_{k}^{2}>B\right)+
\exp\left(\left(-\lambda\delta+\frac{\lambda^{2}B}{2}\right)2n\right).
\end{align}
Optimizing  on $\lambda$ , we get
\begin{equation*}
\pp\left(\frac{V_{n}}{2n}>\delta\right)\leq
\pp\left(\frac{\phi}{2n}\sum\limits_{k\in\ttt_{r_{n},p-1}}\|\xx_{k}\|^{2}>B\right)+
\exp\left(-\frac{\delta^{2}}{B}2n\right).
\end{equation*}
Since the same thing works for $-V_{n}$ instead of $V_{n}$, using
$|\ttt_{n-1}|$ instead of $n$ in the previous inequality, we have particularly
\begin{equation}\label{V1}
\pp\left(\frac{\left|V_{|\ttt_{n-1}|}\right|}{|\ttt_{n}|+1}>\delta\right)\leq
\pp\left(\frac{\phi}{|\ttt_{n}|+1}\sum\limits_{k\in\ttt_{n-1,p-1}}\|\xx_{k}\|^{2}>B\right)+
\exp\left(-\frac{\delta^{2}}{B}2^{n+1}\right).
\end{equation}

Now, to control the first term in the right hand of the last inequality, we
will use the decomposition given by (\ref{decom_crochet_V}).
From the convergence of
$\frac{4\phi}{(1-\beta)(|\ttt_{n}|+1)} P_{n}$ and
$\frac{4\phi\alpha^{2}}{(1-\beta)(|\ttt_{n}|+1)} Q_{n}$ (see
\cite{BSG09} for more details) let $l_{1}$ and $l_{2}$ such that $\forall n\geq p-1$
\[
\frac{4\phi P_{n-1}}{(1-\beta)(|\ttt_{n}|+1)}\rightarrow l_{1}
\hspace{0.50cm} \text{and} \hspace{0.50cm}
\frac{4\phi \alpha^{2}Q_{n-1}}{(1-\beta)(|\ttt_{n}|+1)}<l_{2}.
\]
For $\delta>0$, we choose $B=\delta+l_{1}+l_{2},$ using
(\ref{decom_crochet_V}), we then have
\begin{eqnarray}\label{ineqexpoX2}
&\,&\pp\left(\frac{\phi}{|\ttt_{n}|+1}\sum\limits_{k\in\ttt_{n-1,p-1}}\|\xx_{k}\|^{2}>B\right)\notag\\
&\leq&
\pp\left(\frac{P_{n-1}}{|\ttt_{n}|+1}-l_{1}'>\delta_{1}\right)+\pp\left(\frac{Q_{n-1}}{|\ttt_{n}|+1}-l_{2}'>\delta_{2}\right)+\pp\left(\frac{R_{n-1}\overline{X}_{1}^{2}}{|\ttt_{n}|+1}>\delta_{3}\right)
\end{eqnarray}
where
\[
\delta_{1}=\frac{(1-\beta)\delta}{12\phi},\quad
l_{1}'=\frac{(1-\beta)l_{1}}{4\phi},\quad
\delta_{2}=\frac{(1-\beta)\delta}{12\alpha^{2}\phi},\quad
l_{2}'=\frac{(1-\beta)l_{2}}{4\alpha^{2}\phi},\hspace{0.20cm}
\text{and} \hspace{0.20cm} \delta_{3}=\frac{\delta}{6\phi}.
\]
First, by the choice of $l_{2}$, we have
\begin{equation}\label{V2}
\pp\left(\frac{Q_{n-1}}{|\ttt_{n}|+1}-l_{2}'>\delta_{2}\right)=0.
\end{equation}
Next, from Chernoff inequality and hypothesis {\bf (X2)} we get
easily
\begin{equation}\label{V3}
\pp\left(\frac{R_{n-1}\overline{X}_{1}^{2}}{|\ttt_{n}|+1}>\delta_{3}\right)\leq\begin{cases}
c_{1}\exp\Big(-c_{2}\delta 2^{n+1}\Big) \hspace{0.25cm} \text{if
$\beta<\frac{\sqrt{2}}{2}$}\\ \\
c_{1}\exp\left(-c_{2}\delta\frac{2^{n+1}}{n+1}\right)
\hspace{0.25cm} \text{if $\beta=\frac{\sqrt{2}}{2}$} \\ \\
c_{1}\exp\left(-c_{2}\delta\left(\frac{1}{\beta^{2}}\right)^{n+1}\right)
\hspace{0.25cm} \text{if $\beta>\frac{\sqrt{2}}{2}$},\end{cases}
\end{equation}
for some positive constants $c_{1}$ and $c_{2}$. Let us now control
the first term of the right hand side of (\ref{ineqexpoX2}).
\medskip

{\bf First case.} If $\beta=\frac{1}{2}$, from \cite{BSG09}
\[
P_{n-1}=\sum\limits_{k=p}^{n-1}(n-k)\sum\limits_{i\in\g_{k}}\vep_{i}^{2}
\hspace{0.25cm} \text{and} \hspace{0.25cm} l_{1}'=\sigma^{2}.
\]
We thus have
\[
\frac{P_{n-1}}{|\ttt_{n}|+1}-\sigma^{2}=\frac{1}{|\ttt_{n}|+1}
\sum\limits_{k=p}^{n-1}(n-k)\sum\limits_{i\in\g_{k}}(\vep_{i}^{2}-\sigma^{2})
+\sigma^{2}\left(\sum\limits_{k=p}^{n-1}\frac{n-k}{2^{n+1-k}}
-1\right).
\]
In addition, we also have
\[
\sigma^{2}\left(\sum\limits_{k=p}^{n-1}\frac{n-k}{2^{n+1-k}}
-1\right)\leq 0. \]
We thus deduce that
\[
\pp\left(\frac{P_{n-1}}{|\ttt_{n}|+1}-l_{1}'>\delta_{1}\right) \leq
\pp\left(\frac{1}{|\ttt_{n}|+1}
\sum\limits_{k=p}^{n-1}(n-k)\sum\limits_{i\in\g_{k}}(\vep_{i}^{2} -
\sigma^{2})>\delta_{1}\right).
\]
On the one hand we have
\begin{align}
\notag\pp\left(\frac{1}{|\ttt_{n}|+1} \sum\limits_{k=p}^{n-1}(n-k)
\sum\limits_{i\in\g_{k}}(\vep_{i}^{2}-\sigma^{2})>\delta_{1}\right)
\hspace{7cm}\left.\right.\\ \leq
\sum_{\eta=0}^1\pp\left(\frac{1}{|\ttt_{n}|+1} \sum\limits_{k=p-1}^{n-2}(n-k-1)
\sum\limits_{i\in\g_{k}}(\vep_{2i+\eta}^{2}-\sigma^{2})>\delta_{1}/2\right).\label{negterm}
\end{align}
On the other hand, for all $\lambda>0$, an application of Chernoff
inequality yields
\begin{eqnarray*}
&\,&\pp\left(\frac{1}{|\ttt_{n}|+1}
\sum\limits_{k=p-1}^{n-2}(n-k-1)
\sum\limits_{i\in\g_{k}}(\vep_{2i}^{2}-\sigma^{2})>
\delta_{1}/2\right)\\
&\leq& \exp\left(\frac{-\delta_{1}\lambda2^{n+1}}{2}\right) \times
\ee\left[\exp\left(\lambda\sum\limits_{k=p-1}^{n-2}(n-k-1)
\sum\limits_{i\in\g_{k}}(\vep_{2i}^{2} - \sigma^{2})\right)\right].
\end{eqnarray*}
From hypothesis {\bf (N2)} we get
\begin{eqnarray*}
&\,&\ee\left[\exp\left(\lambda\sum\limits_{k=p-1}^{n-2}(n-k-1)
\sum\limits_{i\in\g_{k}}(\vep_{2i}^{2}-\sigma^{2})\right)\right]\\
&=&\ee\left[\ee\left[\exp\left(\lambda\sum\limits_{k=p-1}^{n-2}(n-k-1)
\sum\limits_{i\in\g_{k}}(\vep_{2i}^{2} - \sigma^{2})\right)\Big/
\FF_{n}\right]\right]\\
&=&
\ee\left[\exp\left(\lambda\sum\limits_{k=p-1}^{n-3}(n-k-1)
\sum\limits_{i\in\g_{k}}(\vep_{2i}^{2}-\sigma^{2})\right)
\prod\limits_{i\in\g_{n-2}}\ee\left[\exp\left(\lambda(\vep_{2i}^{2} - \sigma^{2})\right)\Big/\FF_{n}\right]\right]\\
&\leq& \exp\left(\lambda^{2}\gamma|\g_{n-2}|\right)
\ee\left[\exp\left(\lambda \sum\limits_{k=p-1}^{n-3}(n-k-1)
\sum\limits_{i\in\g_{k}}(\vep_{2i}^{2}-\sigma^{2})\right)\right].
\end{eqnarray*}
Iterating this procedure, we obtain
\begin{eqnarray*}
\ee\left[\exp\left(\lambda \sum\limits_{k=p-1}^{n-2}(n-k-1)
\sum\limits_{i\in\g_{k}}(\vep_{2i}^{2}-\sigma^{2})\right)\right]
&\leq& \exp\left(\gamma\lambda^{2}
\sum\limits_{k=2}^{n-p+1}k^{2}|\g_{n-k}|\right)
\\ &\leq& \exp\left(c\gamma\lambda^{2}2^{n+1}\right),
\end{eqnarray*}
where $c=\sum\limits_{k=1}^{\infty}\frac{k^{2}}{2^{k+2}}.$
Optimizing on $\lambda$, we are led, for some positive constant
$c_{1}$ to
\[
\pp\left(\frac{1}{|\ttt_{n}|+1} \sum\limits_{k=p-1}^{n-2}(n-k-1)
\sum\limits_{i\in\g_{k}}(\vep_{2i}^{2}-\sigma^{2})>
\delta_{1}/2\right) \leq
\exp\left(-c_{1}\delta^{2}|\ttt_{n}|\right).
\]
Following the same lines, we obtain the same inequality for the second term in (\ref{negterm}).
It then follows that
\begin{equation}\label{V4}
\pp\left(\frac{P_{n-1}}{|\ttt_{n}|+1}-l_{1}'>\delta_{1}\right)\leq
c_{1}\exp\left(-c_{2}\delta^{2}|\ttt_{n}|\right),
\end{equation}
for some positive constants $c_{1}$ and $c_{2}.$
\medskip

{\bf Second case.} If $\beta\neq \frac{1}{2},$ then from \cite{BSG09}, we have
$l_{1}'=\frac{\sigma^{2}}{2(1-\beta)}.$ Since
\[
\sigma^{2}\left(\sum\limits_{k=p}^{n-1}\frac{1-(2\beta)^{n-k}}{(1-2\beta)2^{n-k+1}}\right)\leq
\frac{\sigma^{2}}{2(1-\beta)},
\]
we deduce that
\[
\pp\left(\frac{P_{n-1}}{|\ttt_{n}|+1}-l_{1}'>\delta_{1}\right)\leq
\pp\left(\frac{1}{|\ttt_{n}|+1}
\sum\limits_{k=p}^{n-1}\frac{1-(2\beta)^{n-k}}{1-2\beta}
\sum\limits_{i\in\g_{k}}(\vep_{i}^{2}-\sigma^{2})>\delta_{1}\right).
\]

\begin{itemize}
\item If $\beta<\frac{1}{2},$ then for some positive constant $c$ we
have
\[
\pp\left(\frac{P_{n-1}}{|\ttt_{n}|+1}-l_{1}'>\delta_{1}\right)\leq
\pp\left(\frac{1}{|\ttt_{n}|+1}\sum\limits_{k=p}^{n-1}
\sum\limits_{i\in\g_{k}}(\vep_{i}^{2}-\sigma^{2})>c\delta_{1}\right).
\]
Performing now as in the proof of (\ref{conv2}), we get
\begin{equation}\label{V5}
\pp\left(\frac{P_{n-1}}{|\ttt_{n}|+1}-l_{1}'>\delta_{1}\right)\leq
c_{1}\exp\left(-c_{2}\delta^{2}|\ttt_{n}|\right),
\end{equation}
for some positive constants $c_{1}$ and $c_{2}$.

\item If $\beta>\frac{1}{2}$, then for some positive constant $c$,
we have
\[
\pp\left(\frac{P_{n-1}}{|\ttt_{n}|+1}-l_{1}'>\delta_{1}\right)\leq
\pp\left(\frac{1}{|\ttt_{n}|+1}\sum\limits_{k=p}^{n-1}(2\beta)^{n-k}
\sum\limits_{i\in\g_{k}}(\vep_{i}^{2}-\sigma^{2})>c\delta_{1}\right).
\]
Now, from Chernoff inequality, hypothesis {\bf (N2)} and after
several successive conditioning, we get for all $\lambda>0$
\begin{eqnarray*}
&\,&\pp\left(\frac{1}{|\ttt_{n}|+1}
\sum\limits_{k=p}^{n-1}(2\beta)^{n-k}
\sum\limits_{i\in\g_{k}}(\vep_{i}^{2} - \sigma^{2})>c\delta_{1}\right)\\
&\leq& \exp\Big(-c\delta_{1}\lambda2^{n+1}\Big)
\exp\left(\gamma\lambda^{2}2^{n+1}\sum\limits_{k=2}^{n-p+1}(2\beta^{2})^{k}\right).
\end{eqnarray*}

Next, optimizing over $\lambda$, we are led, for some positive
constant $c$ to
\begin{equation}\label{V6}
\pp\left(\frac{P_{n-1}}{|\ttt_{n}|+1}-l_{1}'>\delta_{1}\right) \leq
\begin{cases} \exp\Big(-c\delta^{2}|\ttt_{n}|\Big) \hspace{0.25cm}
\text{if
$\frac{1}{2}<\beta<\frac{\sqrt{2}}{2}$}, \\ \\
\exp\Big(-c\delta^{2}\frac{|\ttt_{n}|}{n}\Big) \hspace{0.25cm}
\text{if $\beta=\frac{\sqrt{2}}{2}$}, \\ \\
\exp\left(-c\delta^{2}\left(\frac{1}{\beta^{2}}\right)^{n+1}\right)
\hspace{0.25cm} \text{if $\beta>\frac{\sqrt{2}}{2}$}.\end{cases}
\end{equation}
\end{itemize}

Now combining (\ref{V1}), (\ref{ineqexpoX2}), (\ref{V2}), (\ref{V3}),
(\ref{V4}),(\ref{V5}) and (\ref{V6}), we have thus showed that

\begin{equation}\label{ineqexpoU_n}
\begin{array}{ll}
\pp\left(\frac{1}{|\ttt_{n}|+1}\left|V_{|\ttt_{n-1}|}\right|>\delta\right)\\ \\ \leq\begin{cases}
c_{1}\exp\left(-c_{2}\delta^{2}2^{n+1}\right)+c_{1}\exp\left(-c_{2}\delta 2^{n+1}\right)+
\exp\left(\frac{-\delta^{2}}{\delta+l_{1}+l_{2}}2^{n+1}\right)
\hspace{0.10cm} \text{if $\beta<\frac{\sqrt{2}}{2}$} ,\\ \\
c_{1}\exp\left(-c_{2}\delta^2\frac{2^{n+1}}{n+1}\right) +
c_{1}\exp\left(-c_{2}\delta\frac{2^{n+1}}{n+1}\right) +
\exp\left(\frac{-\delta^{2}}{\delta+l_{1}+l_{2}}2^{n+1}\right)
\hspace{0.10cm} \text{if $\beta=\frac{\sqrt{2}}{2}$}, \\ \\
c_{1}\exp\left(-c_{2}\delta^{2}\left(\frac{1}{\beta^{2}}\right)^{n+1}\right)
+ c_{1}
\exp\left(-c_{2}\delta\left(\frac{1}{\beta^{2}}\right)^{n+1}\right)
+ \exp\left(\frac{-\delta^{2}}{\delta+l_{1}+l_{2}}2^{n+1}\right)
\hspace{0.10cm} \text{if $\beta>\frac{\sqrt{2}}{2}$},
\end{cases}
\end{array}
\end{equation}
where the positive constants $c_{1}$ and $c_{2}$ may differ term by
term.

One can easily check that the coefficients of the matrix $U_{n}$ are
linear combinations of terms similar to $V_{|\ttt_{n-1}|}$, so that
performing to similar calculations as before for each of them, we
deduce the same deviation inequalities for $U_{n}$ as in (\ref{ineqexpoU_n}).

Now we have
\begin{align}
\notag
\pp\left(\sum\limits_{k=0}^{n-p}\frac{1}{2^{k}}\left\|\sum\limits_{C\in\{A;B\}^{k}}C
\frac{U_{n-k}}{2^{n-k+1}}C^{t}\right\|>\delta\right) & \leq
\pp\left(\sum\limits_{k=0}^{n-p}\frac{1}{2^{k}}\sum\limits_{C\in\{A;B\}^{k}}
\frac{1}{2^{n-k+1}}\left\|CU_{n-k}C^{t}\right\|> \delta\right) \\
\notag & \leq
\pp\left(\sum\limits_{k=p}^{n}\beta^{2(n-k)}\frac{1}{|\ttt_{k}|+1}\|U_{k}\|>\delta\right)\\
\notag & \leq
\sum\limits_{k=p}^{n}\pp\left(\frac{\|U_{k}\|}{|\ttt_{k}|+1}>\frac{\delta}{(n-p+1)\beta^{2(n-k)}}\right).
\end{align}

From (\ref{ineqexpoU_n}), we infer the following
\begin{align}
\notag
\pp\left(\sum\limits_{k=0}^{n-p}\frac{1}{2^{k}}\left\|\sum\limits_{C\in\{A;B\}^{k}}C
\frac{U_{n-k}}{2^{n-k+1}}C^{t}\right\|>\delta\right) \hspace{8cm} \left.\right. \\
\notag \leq
\begin{cases}
c_{1}\sum\limits_{k=p}^{n}\exp\left(-c_{2}\frac{\delta^{2}(2\beta^{4})^{k+1}}{n^{2}\beta^{4n}}\right)
+
c_{1}\sum\limits_{k=p}^{n}\exp\left(-c_{2}\frac{\delta(2\beta^{2})^{k+1}}{n\beta^{2n}}\right)
\\ \notag \hspace{4cm} \left.\right. +
c_{1}\sum\limits_{k=p}^{n}\exp\left(-c_{2}\frac{\delta^{2}2^{k+1}}{(\delta+nl\beta^{2(n-k-1)})n\beta^{2(n-k-1)}}\right)
\hspace{0.25cm} \text{if $\beta<\frac{\sqrt{2}}{2}$} , \\ \\ \notag
c_{1}\sum\limits_{k=p}^{n}\exp\left(-c_{2}\frac{\delta^{2}4^{n}}{n^{2}(k+1)2^{k+1}}\right)
+
c_{1}\sum\limits_{k=p}^{n}\exp\left(-c_{2}\frac{\delta2^{n}}{(k+1)n}\right)
\\ \notag \hspace{4.5cm} +
c_{1}\sum\limits_{k=p}^{n}\exp\left(-c_{2}\frac{\delta^{2}2^{k+1}}{(\delta+nl2^{-(n-k-1)})n2^{-(n-k-1)}}\right)
\hspace{0.25cm} \text{if $\beta=\frac{\sqrt{2}}{2}$}, \\ \\ \notag
c_{1}\sum\limits_{k=p}^{n}\exp\left(-c_{2}\frac{\delta^{2}(2\beta^{2})^{k+1}}{n^{2}\beta^{4n}}\right)
+
c_{1}\sum\limits_{k=p}^{n}\exp\left(-c_{2}\frac{\delta}{n\beta^{2n}}\right)
\\ \notag \hspace{4.3cm} +
c_{1}\sum\limits_{k=p}^{n}\exp\left(-c_{2}\frac{\delta^{2}2^{k+1}}{(\delta+nl\beta^{2(n-k-1)})n\beta^{2(n-k-1)}}\right)
\hspace{0.25cm} \text{if $\beta>\frac{\sqrt{2}}{2}$}, \end{cases}
\end{align}
where $l=l_{1}+l_{2}$ and the positive constants $c_{1}$ and $c_{2}$
may differ term by term.

Now
\begin{itemize}
\item If $\beta<\frac{\sqrt{2}}{2}$, then on the one hand,
\begin{align}
\notag\sum\limits_{k=p}^{n}\exp\left(-c\frac{\delta^{2}(2\beta^{4})^{k+1}}{n^{2}\beta^{4n}}\right)
\hspace{10cm} \left.\right.
\\ \notag =\exp\left(-c\delta^{2}\beta^{4}\frac{2^{n+1}}{n^{2}}\right)\left(1+\sum\limits_{k=p}^{n-1}
\left(\exp\left(\frac{-c\delta^{2}}{n^{2}}\right)\right)^{(2\beta^{4})^{k+1}\beta^{-4n}(1-(2\beta^{4})^{n-k})}\right)
\\ \notag \leq
\exp\left(-c\delta^{2}\beta^{4}\frac{2^{n+1}}{n^{2}}\right)\Big(1+o(1)\Big),
\hspace{6.5cm} \left.\right.
\end{align}
where the last inequality follows from the fact that for some
positive constant $c_{1}$,
\[
(2\beta^{4})^{k+1}\beta^{-4n}(1-(2\beta^{4})^{n-k})\propto
c_{1}(2\beta^{4})^{k+1}\beta^{-4n}.
\]
On the other hand, following the same lines as before, we obtain
\begin{eqnarray*}
\sum\limits_{k=p}^{n}\exp\left(-\frac{\delta^{2}2^{k+1}}{(\delta+ln\beta^{2(n-k-1)})n\beta^{2(n-k-1)}}\right)
&\leq&
\sum\limits_{k=p}^{n}\exp\left(-c\delta^{2}\frac{2^{k+1}}{n^{2}\beta^{2(n-k-1)}}\right)
\\ &\leq& \exp\left(-c\frac{\delta^{2}2^{n+1}}{(\delta + l)n^{2}}\right)\Big(1+o(1)\Big),
\end{eqnarray*}
and
\begin{eqnarray*}\sum\limits_{k=p}^{n}\exp\left(-c\frac{\delta(2\beta^{2})^{k+1}}{n\beta^{2n}}\right)
&\leq&
\sum\limits_{k=p}^{n}\exp\left(-c\frac{\delta(2\beta^{2})^{k+1}}{n^{2}\beta^{2n}}\right)\\
& \leq&
\exp\left(-c\delta\frac{2^{n+1}}{n^{2}}\right)\Big(1+o(1)\Big).
\end{eqnarray*}

We thus deduce that
\begin{equation}\label{U1}\pp\left(\sum\limits_{k=0}^{n-p}\frac{1}{2^{k}}\left\|\sum\limits_{C\in\{A;B\}^{k}}C
\frac{U_{n-k}}{2^{n-k+1}}C^{t}\right\|>\delta\right) \leq
c_{1}\exp\left(-c_{2}\delta^{2}\frac{2^{n+1}}{n^{2}}\right) +
c_{1}\exp\left(-c_{2}\delta\frac{2^{n+1}}{n^{2}}\right),
\end{equation}
for some positive constants $c_{1}$ and $c_{2}$.

\item If $\beta=\frac{\sqrt{2}}{2}$, then following the same
lines as before, we show that
\begin{align}
\notag\sum\limits_{k=p}^{n}\exp\left(-c\delta^{2}\frac{4^{n}}{n^{2}(k+1)2^{k+1}}\right)\leq
\exp\left(-c\delta^{2}\frac{2^{n+1}}{n^{3}}\right)\Big(1+o(1)\Big),
\\ \notag
\sum\limits_{k=p}^{n}\exp\left(-\frac{\delta^{2}2^{k+1}}{(\delta+ln2^{-(n-k-1)})n2^{-(n-k-1)}}\right)
\leq \exp\left(-c\frac{\delta^{2}2^{n+1}}{n^{2}(\delta + l)}\right)\Big(1+o(1)\Big), \\
\notag
\sum\limits_{k=p}^{n}\exp\left(-c\delta\frac{2^{n}}{n(k+1)}\right)\leq
\exp\left(-c\delta\frac{2^{n+1}}{n^{3}}\right)\Big(1+o(1)\Big).
\end{align}
It then follows that
\begin{eqnarray}\label{U2}
&\,&\pp\left(\sum\limits_{k=0}^{n-p}\frac{1}{2^{k}}\left\|\sum\limits_{C\in\{A;B\}^{k}}C
\frac{U_{n-k}}{2^{n-k+1}}C^{t}\right\|>\delta\right) \\
&\leq& c_{1}\exp\left(-c_{2}\delta^{2}\frac{2^{n+1}}{n^{3}}\right) +
c_{1}\exp\left(-c_{2}\frac{\delta^{2}2^{n+1}}{n^{2}(\delta +
l)}\right)+ c_{1}\exp\left(-c_{2}\delta\frac{2^{n+1}}{n^{3}}\right),
\end{eqnarray}
for some positive constants $c_{1}$ and $c_{2}$.

\item If $\beta>\frac{\sqrt{2}}{2},$ once again following the
previous lines, we get
\begin{eqnarray}\label{U3}
&\,&\pp\left(\sum\limits_{k=0}^{n-p}\frac{1}{2^{k}}\left\|\sum\limits_{C\in\{A;B\}^{k}}C
\frac{U_{n-k}}{2^{n-k+1}}C^{t}\right\|>\delta\right)\notag\\
 &\leq&
c_{1}\exp\left(-c_{2}\delta^{2}\frac{1}{n^{2}\beta^{2n}}\right) +
c_{1}\exp\left(-c_{2}\frac{\delta^{2}}{(\delta +
l)n^{2}\beta^{2n}}\right)+ c_{1}n\exp\left(-c_{2}\frac{\delta}
{n^{2}\beta^{2n}}\right)
\end{eqnarray}
for some positive constants $c_{1}$ and $c_{2}$.
\end{itemize}

We infer from the inequalities (\ref{U1}), (\ref{U2})and  (\ref{U3}) that
\[
\sum\limits_{k=0}^{n-p}\frac{1}{2^{k}}\sum\limits_{C\in\{A;B\}^{k}}C
\frac{U_{n-k}}{2^{n-k+1}}C^{t} \superexpn 0.
\]
\end{proof}
This achieves the proof of the Proposition \ref{convexpocrochet2}.
\end{proof}
We now, explain the modification in the last proofs in the case 1.
\begin{prop}\label{convexpocrochet3}
Within the framework 1, we have the same conclusions as the
Proposition \ref{convexpocrochet1} and \ref{convexpocrochet2} with
the sequence $(b_{n})$ which satisfies condition {\bf (V1)}.
\end{prop}

\begin{proof}
The proof follows exactly the same lines as the proof of Proposition
\ref{convexpocrochet1} and \ref{convexpocrochet2}, and uses the fact
that if a superexponential convergence holds with a sequence
$(b_{n})$ which satisfies condition {\bf (V2)}, then it also holds
with a sequence $(b_{n})$ which satisfies condition {\bf (V1)}. We
thus obtain the first convergence of (\ref{conv_sum_H_p1}), the
convergences (\ref{conv1}), (\ref{conv4_2}), (\ref{conv4_3}) and
(\ref{conv3}) within the framework 1 with $(b_{n})$ which satisfies
condition {\bf (V1)}. Next, following the same approach as which
used to obtain (\ref{devineqeps2}), we get
\begin{equation}\label{devineqeps3}
\pp\left(\frac{1}{|\ttt_{k}|+1} \left|\sum\limits_{i\in\ttt_{k,p}}
(\vep_{i}^{2}-\sigma^{2})\right| >\delta\right) \leq \begin{cases}
c_{1} \exp\left(-c_{2}\delta^{2}|\ttt_{k}|\right) \hspace{0.25cm}
\text{if $\delta$ is small enough} \\ c_{1}
\exp\left(-c_{2}\delta|\ttt_{k}|\right) \hspace{0.25cm} \text{if
$\delta$ is large enough}, \end{cases}
\end{equation}
where $c_{1}$ and $c_{2}$ are positive constants which do not depend
on $\delta.$ The first inequality holds for example if
$\delta/\gamma < \varepsilon$ and the second holds for example if
$\delta/\gamma > \varepsilon.$ On the other hand, for $n$ large
enough, let $n_{0}$ such that for all $k < n_{0},$ $n\beta^{2(n-k)}$
is small enough so that $\delta/(n-p+1)\gamma\beta^{2(n-k)} >
\veps.$ We have
\begin{align*}
\pp\left(\left\|\sum\limits_{k=0}^{n-p} \frac{1}{2^{k}}
\sum\limits_{C\in\{A;B\}^{k}} C\frac{L_{n-k}-\sigma^{2}}{2^{n-k}}
e_{1}e_{1}^{t}C^{t}\right\|>\delta\right) \hspace{7cm} \left.\right.
\\ \leq \sum\limits_{k=p}^{n_{0}-1}
\pp\left(\frac{|L_{k}-\sigma^{2}|}{|\ttt_{k}|+1}
> \frac{\delta}{(n-p+1)\beta^{2(n-k)}}\right) + \sum\limits_{k=n_{0}}^{n}
\pp\left(\frac{|L_{k}-\sigma^{2}|}{|\ttt_{k}|+1}
> \frac{\delta}{(n-p+1)\beta^{2(n-k)}}\right).
\end{align*}
Now, using (\ref{devineqeps3}) with $\delta/(n-p+1)\beta^{2(n-k)}$
instead of $\delta$ and following the same approach used to obtain
(\ref{U1})-(\ref{U3}) in the two sums of the right hand side of the
above inequality, we are led to
\begin{align*}
\pp\left(\left\|\sum\limits_{k=0}^{n-p} \frac{1}{2^{k}}
\sum\limits_{C\in\{A;B\}^{k}} C\frac{L_{n-k}-\sigma^{2}}{2^{n-k}}
e_{1}e_{1}^{t}C^{t}\right\|>\delta\right) \hspace{7cm} \left.\right.
\\ \leq \begin{cases} c_{1} \exp\left(-
\frac{c_{2}\delta^{2}2^{n+1}}{n^{2}}\right) + c_{1} \exp\left(-
\frac{c_{2}\delta 2^{n+1}}{n}\right) \quad \text{if $\beta\leq
\frac{1}{2}$} \\ c_{1}n\exp\left(-
\frac{c_{2}\delta^{2}}{n^{2}\beta^{4n}}\right) + c_{1} \exp\left(-
\frac{c_{2}\delta}{n\beta^{2n}}\right) \quad \text{if
$\beta>\frac{1}{2}$}, \end{cases}
\end{align*}
and we thus obtain convergence (\ref{conv2}) with $(b_{n})$ which
satisfies condition {\bf (V1)}. In the same way we obtain
\[
\pp\left(\|T_{n}^{(3)}\| > \delta\right) \leq \begin{cases} c_{1}
\exp\left(- \frac{c_{2}\delta^{2}2^{n+1}}{n^{2}}\right) + c_{1}
\exp\left(- \frac{c_{2}\delta2^{n+1}}{n}\right) \quad \text{if
$\beta<\frac{1}{2}$}, \\ c_{1}n
\exp\left(-\frac{c_{2}\delta2^{n+1}}{n}\right) \hspace{4cm} \text{if
$\beta=\frac{1}{2}$}, \\ c_{1}
\exp\left(-\frac{c_{2}\delta^{2}}{n^{2}\beta^{2n}}\right) + c_{1}
\exp\left(-\frac{c_{2}\delta}{n\beta^{n}}\right) \quad \text{if
$\beta>\frac{1}{2}$}, \end{cases}
\]
so that (\ref{conv4_1}) and then (\ref{conv4}) hold for $(b_{n})$
which satisfies condition {\bf (V1)}. To reach the convergence
(\ref{conv5}) and the second convergence of (\ref{conv_sum_H_p1})
with $(b_{n})$ which satisfies condition {\bf (V1)}, we follow the
same procedure as before and the proof of proposition is then
complete.
\end{proof}
\begin{rem}\label{convexpocrochet_incomplet}
Let us note that we can actually prove that
\[
\frac{1}{n}\sum\limits_{k=2^{p}}^{n}\xx_{k} \superexpnn
L_{1,2}\qquad \text{and} \qquad
\frac{1}{n}\sum\limits_{k=2^{p}}^{n}\xx_{k}\xx_{k}^{t} \superexpnn
L_{2,2}.
\]
Indeed, let $\displaystyle H_n=\sum_{k=2^{p-1}}^{n}\xx_k\quad{\rm and}\quad
P_l^{(n)}=\sum_{k=2^{r_{n}-l}}^{[\frac{n}{2^{l}}]}\vep_k.
$
We have the following decomposition
\[
\frac{H_{n}}{n} = \frac{1}{n}\sum\limits_{k\in
\ttt_{r_{n}-1,p-1}}\xx_{k} + \frac{1}{n}
\sum\limits_{k=2^{r_{n}}}^{n}\xx_{k}.
\]
On the one hand, from Proposition \ref{convexpocrochet1}, we infer that
\[
\frac{1}{n}\sum\limits_{k\in \ttt_{r_{n}-1,p-1}}\xx_{k} \superexpnn
cL_{1,2},
\]
where $c = \underset{n\rightarrow\infty}\lim\frac{2^{r_{n}}-1}{n}.$

On the other hand, from (\ref{bar_p2}) we deduce that
\begin{align}
\notag\sum\limits_{k=2^{r_{n}}}^{n}\xx_{k} =
2^{r_{n}-p+1}\left(\overline{A}\right)^{r_{n}-p+1}
\sum\limits_{k=2^{p-1}}^{[\frac{n}{2^{r_{n}}-p+1}]} \xx_{k} +
2\overline{a}\sum\limits_{k=0}^{r_{n}-p}\left([\frac{n}{2^{k}}]  -
2^{r_{n}-k} + 1\right)2^{k}\left(\overline{A}\right)^{k}e_{1} \\
\notag +
\sum\limits_{k=0}^{r_{n}-p}2^{k}\left(\overline{A}\right)^{k}P_{k}^{(n)}e_{1}
-
\sum\limits_{k=1}^{r_{n}-p+1}s_{k}2^{k-1}\left(\overline{A}\right)^{k-1}\left(B\xx_{[\frac{n}{2^{k}}]}
+ \eta_{[\frac{n}{2^{k-1}}]+1}\right),
\end{align}
where
\[
s_{k}=\begin{cases} 1 \hspace{0.25cm} \text{if $[\frac{n}{2^{k-1}}]$ is
even} \\ 0 \hspace{0.25cm} \text{if $[\frac{n}{2^{k-1}}]$ is
old}.\end{cases}
\]
Performing now as in the proof of Proposition
\ref{convexpocrochet1}, tedious but straightforward calculations
lead us to
\[
\frac{1}{n} \sum\limits_{k=2^{r_{n}}}^{n}\xx_{k} \superexpnn
(1-c)L_{1,2}
\]
and it then follows that
\[
\frac{1}{n}\sum\limits_{k=2^{p}}^{n}\xx_{k} \superexpnn L_{1,2}.
\]
The term $\frac{1}{n}\sum\limits_{k=2^{p}}^{n}\xx_{k}\xx_{k}^{t}$
can be dealt with in the same way.
\end{rem}
The rest of the paper is dedicated to the proof of our main results. We focus on the proof in the case 2, and some explanation are given on how  to obtain the results in the case 1.

\section{Proof of the main results}
We start with the proof of the  deviation inequalities.
\subsection{Proof of Theorem \ref{thm:deviation_ineq_theta}}
We begin the proof with the case 2. Let $\delta>0$ and $b>0$ such
that $b < \|\Sigma\|/(1 + \delta)$. We have from (\ref{thetaM_n})
\begin{align*}
\pp\left(\|\hat{\theta}_{n} - \theta\| > \delta\right) &=
\pp\left(\frac{\|M_{n}\|}{\|\Sigma_{n-1}\|} > \delta,
\frac{\|\Sigma_{n-1}\|}{|\ttt_{n-1}|} \geq b\right) +
\pp\left(\frac{\|M_{n}\|}{\|\Sigma_{n-1}\|} > \delta,
\frac{\|\Sigma_{n-1}\|}{|\ttt_{n-1}|} < b\right) \\ &\leq
\pp\left(\frac{\|M_{n}\|}{|\ttt_{n-1}|} > \delta b\right) +
\pp\left(\left\|\frac{\Sigma_{n-1}}{|\ttt_{n-1}|} - \Sigma\right\|
>\|\Sigma\|- b\right).
\end{align*}
Since $b<\|\Sigma\|/(1+\delta)$, then,
\[
\pp\left(\left\|\frac{\Sigma_{n-1}}{|\ttt_{n-1}|} - \Sigma\right\|
>\|\Sigma\|- b\right) \leq \pp\left(\left\|\frac{\Sigma_{n-1}}{|\ttt_{n-1}|} - \Sigma\right\|
> \delta b \right).
\]

It then follows that
\[
\pp\left(\|\hat{\theta}_{n} - \theta\| > \delta\right) \leq
2\max\left\{\pp\left(\frac{\|M_{n}\|}{|\ttt_{n-1}|} > \delta
b\right), \pp\left(\left\|\frac{\Sigma_{n-1}}{|\ttt_{n-1}|} -
\Sigma\right\|
> \delta b \right)\right\}.
\]
On the one hand, we have
\begin{align*}
\pp\left(\frac{\|M_{n}\|}{|\ttt_{n-1}|} > \delta b\right) \leq
\sum\limits_{\eta=0}^{1} \left\{
\pp\left(\left|\frac{1}{|\ttt_{n-1}|} \sum\limits_{k\in
\ttt_{n-1,p-1}} \vep_{2k+ \eta}\right| > \frac{\delta
b}{4}\right)\right. \left.\right. \hspace{3cm} \\  +
\left.\pp\left(\left\|\frac{1}{|\ttt_{n-1}|} \sum\limits_{k\in
\ttt_{n-1,p-1}} \vep_{2k+\eta} \xx_{k}\right\| > \frac{\delta
b}{4}\right)\right\}.
\end{align*}
Now, by carrying out the same calculations as those which have
permit us to obtain Lemma \ref{lconv3} and equation
(\ref{ineqexpoU_n}), we are led to
\begin{equation}\label{dev_M_n1}
\pp\left(\frac{\|M_{n}\|}{|\ttt_{n-1}|} > \delta b\right) \leq
\begin{cases} c_{1} \exp\left(-\frac{c_{2}(\delta b)^{2}}{c_{3} +
c_{4}(\delta b)}2^{n}\right) \hspace{1.75cm} \text{if $\beta<
\frac{\sqrt{2}}{2}$} ,\\ \\ c_{1}\exp\left(-\frac{c_{2}(\delta
b)^{2}}{c_{3} + c_{4}(\delta b)} \frac{2^{n}}{n}\right)
\hspace{1.7cm} \text{if $\beta = \frac{\sqrt{2}}{2}$}, \\ \\
c_{1}\exp\left(-\frac{c_{2}(\delta b)^{2}}{c_{3} + c_{4}(\delta b)}
\left(\frac{1}{\beta^{2}}\right)^{n}\right) \hspace{0.8cm} \text{if
$\beta> \frac{\sqrt{2}}{2}$},
\end{cases}
\end{equation}
where positive constants $c_{1}$, $c_{2}$, $c_{3}$ and $c_{4}$
depend on $\sigma,$ $\beta,$ $\gamma$ and $\phi$ and $(c_{3}, c_{4})
\neq (0,0).$

On the other hand, noticing that $\Sigma_{n-1} = I_{2}\otimes
S_{n-1}$, we have
\[
\pp\left(\left\|\frac{\Sigma_{n-1}}{|\ttt_{n-1}|} - \Sigma\right\|
> \delta b\right) \leq 2\pp\left(\left\|\frac{S_{n-1}}{|\ttt_{n-1}|} - L\right\|
> \frac{\delta b}{2}\right).
\]
Next, from the proofs of Propositions \ref{convexpocrochet1} and
\ref{convexpocrochet2}, we deduce that
\begin{equation}\label{dev_M_n2}
\pp\left(\left\|\frac{\Sigma_{n-1}}{|\ttt_{n-1}|} - \Sigma\right\|
> \frac{b}{2}\right) \leq \begin{cases} c_{1} \exp\left(-\frac{c_{2}(\delta b)^{2}}{c_{3} +
c_{4}(\delta b)}\frac{2^{n}}{(n-1)^{2}}\right) \hspace{1.6cm}
\text{if $\beta< \frac{\sqrt{2}}{2}$} \\
\\ c_{1}\exp\left(-\frac{c_{2}(\delta b)^{2}}{c_{3} + c_{4}(\delta b)}
\frac{2^{n}}{(n-1)^{3}}\right) \hspace{1.6cm} \text{if $\beta =
\frac{\sqrt{2}}{2}$} \\ \\ c_{1}\exp\left(-\frac{c_{2}(\delta
b)^{2}}{c_{3} + c_{4}(\delta b)}
\left(\frac{1}{(n-1)^{2}\beta^{2n}}\right)\right) \hspace{0.35cm}
\text{if $\beta> \frac{\sqrt{2}}{2}$},
\end{cases}
\end{equation}
where positive constants $c_{1}$, $c_{2}$, $c_{3}$ and $c_{4}$
depend on $\sigma,$ $\beta,$ $\gamma$ and $\phi$ and $(c_{3}, c_{4})
\neq (0,0).$ Now, (\ref{dev_ineq_theta1}) follows from
(\ref{dev_M_n1}) and (\ref{dev_M_n2}).

In the case 1, the proof follows exactly the same lines as before
and uses the same ideas as the proof of Proposition
\ref{convexpocrochet3}. Particularly, we have in this case
\begin{equation*}
\pp\left(\left\|\frac{\Sigma_{n-1}}{|\ttt_{n-1}|} - \Sigma\right\|
> \frac{b}{2}\right) \leq \begin{cases} c_{1} \exp\left(-\frac{c_{2}(\delta b)^{2}}{c_{3} + (\delta b)}
\frac{2^{n}}{(n-1)^{2}}\right) \hspace{0.25cm} \text{if $\beta<
\frac{1}{2}$} \\ \\ c_{1}(n-1)\exp\left(-\frac{c_{2}(\delta
b)^{2}}{c_{3} + (\delta b)} \frac{2^{n}}{(n-1)^{2}}\right)
\hspace{0.25cm} \text{if $\beta = \frac{1}{2}$} \\ \\
c_{1}(n-1)\exp\left(-\frac{c_{2}(\delta b)^{2}}{c_{3} + (\delta b)}
\left(\frac{1}{(n-1)\beta^{n}}\right)\right) \hspace{0.25cm}
\text{if $\beta> \frac{1}{2}$}.
\end{cases}
\end{equation*}
where positive constants $c_{1}$, $c_{2}$ and $c_{3}$ depend on
$\sigma,$ $\beta,$ $\gamma$ and $\phi.$ (\ref{dev_ineq_theta1}) then
follows in this case, and this ends the proof of Theorem.
\ref{thm:deviation_ineq_theta}.


\subsection{Proof of Theorem  \ref{mdp_sigma_n_rho_n}}

At first we need to prove the following

\bthm\label{thm:mdp_M_n} In the case 1 or in the case 2, the sequence $\displaystyle\left(M_n /\left(b_{|\ttt_{n-1}|}\sqrt{|\ttt_{n-1}}|\right)\right)_{n\ge 1}$
satisfies the MDP on $\rr^{2(p+1)}$ with speed
$b^2_{|\ttt_{n-1}|}$ and rate function
\begin{equation}\label{IM}
I_M(x)=\sup_{\lambda\in \rr^{2(p+1)}}\{\lambda^{t} x-\lambda^{t} (\Gamma\otimes
L)\lambda\}=\frac 12 x^t  (\Gamma\otimes L)^{-1} x.
\end{equation}
\nthm

\subsubsection{Proof of Theorem \ref{thm:mdp_M_n} }

Now, as in Bercu et al. \cite{BSG09}, denote by $(\GG_{n})_{n\geq
1}$ the sister pair-wise filtration, that is
$\GG_{n}=\sigma\{X_{1},(X_{2k},X_{2k+1}),1\leq k\leq n\}.$ We
introduce the following $(\GG_n)$ martingale difference sequence
$(D_n)$, given by
\[
D_n=V_n\otimes
Y_n=\begin{pmatrix}
\vep_{2n}\\
\vep_{2n}\xx_n\\
\vep_{2n+1}\\
\vep_{2n+1}\xx_n
\end{pmatrix}.
\]
We clearly have
\[
D_nD_n^t=V_nV_n^t\otimes Y_nY_n^t.
\]
So we obtain that the quadratic variation of the  $(\GG_n)$
martingale $(N_n)_{n\ge 2^{p-1}}$ given by
\[
N_{n} = \sum\limits_{k=2^{p-1}}^{n} D_{k}
\]
is
\[
\langle N\rangle_n=\sum_{k=2^{p-1}}^{n}
\ee(D_kD_k^t/{\GG_{k-1}})=\Gamma\otimes \sum\limits_{k=2^{p-1}}^{n}
Y_{k}Y_{k}^{t}.
\]
Now we clearly have $M_{n} = N_{|\ttt_{n-1}|}$ and $\langle
M\rangle_{n} = \langle N\rangle_{|\ttt_{n-1}|} = \Gamma \otimes
S_{n-1}.$ From Proposition \ref{convergence_crochet}, and since
$\langle M \rangle_n = \Gamma\otimes S_{n-1}$, we have
\begin{equation}
\label{croch}
\frac{\langle M \rangle_n}{|\ttt_n|} \superexp \Gamma\otimes L.
\end{equation}
Before going to the proof of the MDP results, we state the
exponential Lyapounov condition for $(N_{n})_{n\geq 2^{p-1}}$, which
implies exponential Lindeberg condition, that is
\[
\limsup \frac{1}{b_{n}^{2}} \log \pp\left(\frac{1}{n}
\sum\limits_{k=2^{p-1}}^{n}
\ee\left[\|D_{k}\|^{2}\mathbf{1}_{\left\{\|D_{k}\|\geq r
\frac{\sqrt{n}}{b_{n}}\right\}}\right] \geq \delta\right) = -\infty,
\]
(see e.g \cite{Wor01c} for more details on this implication).

\brmk By \cite{DjGuWu06}, we infer from the condition {\bf (Ea)} that
\begin{enumerate}
\item[{\bf (Na)}] one can find $\gamma_{a}>0$ such that
for all $n\geq p-1,$ for all $k\in \g_{n+1}$ and for all $t\in \rr,$ with  $\mu_a=\ee(|\vep_{k}|^{a}/\FF_{n})$ a.s.

\[
\ee\left[\exp t \left(|\vep_{k}|^{a}-\mu_a\right)/\FF_{n}\right] \leq \exp\left(
\frac{\gamma_{a} t^{2}}{2}\right)\quad a.s.
\]
\end{enumerate}
\nrmk
\begin{prop}\label{lemma:lyapounovexpo} Let $(b_n)$ a sequence satisfying the Assumption {\bf (V2)}. Assume that hypothesis {\bf (Na)} and {\bf (Xa)} are satisfied. Then
there exists $B>0$ such that
\[
\limsup_{n\rightarrow \infty}\frac{1}{b_{n}^{2}}
\log\pp\left(\frac{1}{n}
\sum\limits_{j=2^{p-1}}^{n}\ee\left[\|D_{j}\|^{a}/\GG_{j-1}\right]>B\right)=-\infty.
\]
\end{prop}

\begin{proof} We are going to prove that
\begin{equation}\label{expo1}
\limsup_{n\rightarrow \infty}\frac{1}{b_{|\ttt_{n}|}^{2}}
\log\pp\left(\frac{1}{|\ttt_{n}|}
\sum\limits_{j=2^{p}}^{|\ttt_{n}|}\ee\left[\|D_{j}\|^{a}/\GG_{j-1}\right]>B\right)=-\infty,
\end{equation}
and the Proposition (\ref{lemma:lyapounovexpo}) will follow performing as in Remark
\ref{convexpocrochet_incomplet}.  We have
\[
\sum\limits_{j\in\ttt_{n,p}}\ee\left[\|D_{j}\|^{a}/\GG_{j-1}\right]
\leq c\mu^{a}\sum\limits_{j\in
\ttt_{n,p}}\left(1+\|\xx_{j}\|^{a}\right),
\]
where $c$ is a positive constant which depends on $a.$ From
(\ref{bar_p2}), we deduce that
\[
\sum\limits_{j\in\ttt_{n,p}}\|\xx_{j}\|^{a}\leq
\frac{c^2}{(1-\beta)^{a-1}} P_n+
\frac{c^2\alpha^{a}Q_{n}}{(1-\beta)^{a-1}} +
2cR_n\overline{X}_{1}^{a},
\]
where
\[
P_{n}=\sum\limits_{j\in\ttt_{n,p}}
\sum\limits_{i=0}^{r_{j}-p}\beta^{i}|\vep_{[\frac{j}{2^{i}}]}|^{a},
\hspace{0.25cm} Q_{n} = \sum\limits_{j\in\ttt_{n,p}}
\sum\limits_{i=0}^{r_{j}-p}\beta^{i}, \hspace{0.25cm} R_{n} =
\sum\limits_{j\in\ttt_{n,p}}\beta^{a(r_{j}-p+1)},
\]
and $c$ is a positive constant. Now, performing as in the proof of
Proposition \ref{convexpocrochet2},  using hypothesis {\bf (Na)} and
{\bf (Xa)} instead of {\bf (N2)} and {\bf (X2)} we get for $B$ large
enough
\begin{equation}\label{expo3}
\limsup_{n\rightarrow \infty}\frac{1}{b_{|\ttt_{n}|}^{2}}
\log\pp\left(\frac{1}{|\ttt_{n}|}\sum\limits_{j\in
\ttt_{n,p}}\|\xx_{j}\|^{a}>B\right)=-\infty.
\end{equation}
Now (\ref{expo3}) leads us to (\ref{expo1}) and performing as in
Remark \ref{convexpocrochet_incomplet}, we obtain the Proposition \ref{lemma:lyapounovexpo}.

\brmk\label{Markov}In the case 1, we clearly have that $(\xx_{n},
n\in \ttt_{\cdot,p-1}),$ where
\[
\ttt_{\cdot,p-1} = \bigcup\limits_{r=p-1}^{\infty} \g_{r},
\]
is a bifurcating Markov chain with initial state $\xx_{2^{p-1}} =
(X_{2^{p-1}}, X_{2^{p-2}}, \cdots, X_{1})^{t}.$ Let $\nu$ the law of
$\xx_{2^{p-1}}.$ From hypothesis {\bf (X2)}, we deduce that $\nu$
has finite moments of all orders. We denote by $P$ the transition
probability kernel associated to $(\xx_{n}, n\in \ttt_{\cdot,p-1}).$
Let $(\yy_r,r\in \nn)$ the ergodic stable Markov chain associated to
$(\xx_{n}, n\in \ttt_{\cdot,p-1}).$ This Markov chain is defined as
follows, starting from the root $\yy_0=\xx_{2^{p-1}}$ and if
$\yy_r=\xx_n$ then $\yy_{r+1}=\xx_{2n+\zeta_{r+1}}$ for a sequence
of independent Bernoulli r.v. $(\zeta_q, q\in \nn^*)$ such that
$\pp(\zeta_q=0)=\pp(\zeta_q=1)=1/2$. Let $\mu$ the stationary
distribution associated to $(\yy_r,r\in \nn).$ For more details on
bifurcating Markov chain and the associated ergodic stable Markov
chain, we refer to \cite{G07} (see also \cite{BDG11}).

From \cite{BDG11}, we deduce that for all real bounded function $f$
defined on $(\dR^{p})^{3},$
\[
\frac{1}{b_{|\ttt_{n-1}|}\sqrt{|\ttt_{n-1}|}} \sum\limits_{k\in
\ttt_{n-1,p-1}} f\left(\xx_{k}, \xx_{2k}, \xx_{2k+1}\right)
\]
satisfies a MDP on $\dR$ with speed $b_{|\ttt_{n-1}|}^{2}$ and the
rate function $I(x) = \frac{x^{2}}{2S^{2}(f)},$ where $S^{2}(f) =
<\mu, P(f^{2})-(Pf)^{2}>.$

Now, let the function $f$ defined on $(\dR^{p})^{3}$ by $f(x,y,z) =
\|x\|^{2} + \|y\|^{2} + \|z\|^{2}.$ Then, using the relation
(\ref{convergence_crochet}) in Proposition
\ref{prop:convergence_crochet}, the above MDP for real bounded
functionals of the bifurcating Markov chain $(\xx_{n}, n\in
\ttt_{\cdot,p-1})$ and the truncation of the function $f$, we prove
(in the same manner as the proof of lemma 3 in Worms \cite{Wor99d})
that for all $r>0$
\begin{align*}
\limsup_{R\rightarrow \infty}\limsup_{n\rightarrow
\infty}\frac{1}{b_{n}^{2}} \log\pp\Bigg(\frac{1}{n}
\sum\limits_{j=2^{p-1}}^{n} \left(\|X_{j}\|^{2} + \|X_{2j}\|^{2} +
\|X_{2j+1}\|^{2}\right)\Bigg. \hspace{4cm} \left.\right. \\ \Bigg.
\times \mathrm{\bf 1}_{\left\{||\xx_{j}|| + ||\xx_{2j}|| +
||\xx_{2j+1}|| > R \right\}}>r\Bigg) = -\infty,
\end{align*}
which implies the following Lindeberg condition (for more detail one
can see Proposition 2 in Worms \cite{Wor99d})
\begin{align*}
\limsup_{n\rightarrow \infty}\frac{1}{b_{n}^{2}}
\log\pp\Bigg(\frac{1}{n} \sum\limits_{j=2^{p-1}}^{n}
\left(\|X_{j}\|^{2} + \|X_{2j}\|^{2} + \|X_{2j+1}\|^{2}\right)\Bigg.
\hspace{4cm} \left.\right. \\ \Bigg. \times
\mathrm{\bf 1}_{\left\{||\xx_{j}|| + ||\xx_{2j}|| + ||\xx_{2j+1}|| > r
\frac{\sqrt{n}}{b_{n}} \right\}}>\delta \Bigg) = -\infty,
\end{align*}
for all $\delta>0$ and for all $r>0.$ Notice that the above
Lindeberg condition implies particularly the Lindeberg condition on
the sequence $(\xx_{n}).$

\nrmk


Now, we back to the proof of Theorem \ref{thm:mdp_M_n}. We divide the proof into four steps.
In the first one, we  introduce a truncation of the martingale $(M_n)_{n\ge 0}$ and prove that the
truncated martingale satisfies some MDP thanks to Puhalskii's
Theorem \ref{mdppuh}. In the second part, we show that the truncated
martingale is an exponentially good approximation of $(M_n)$, see
e.g. Definition 4.2.14 in \cite{DemZei98}. We conclude by the
identification of the rate function.
\bigskip

{\bf Proof in the case 2}

{\bf Step 1.}  From now on, in order to apply Puhalskii's result  \cite{Puha97}
(Puhalskii's Theorem \ref{mdppuh}) for the MDP for martingales, we introduce the following
truncation of the martingale $(M_n)_{n\ge 0}$. For $r>0$ and $R>0$,
\[
M_n^{(r,R)}=\sum_{k\in \ttt_{n-1,p-1}}D^{(r,R)}_{k,n}.
\]

where, for all $1 \leq k \leq n$,  $D_{k,n}^{(r,R)}=V_k^{(R)}\otimes Y_{k,n}^{(r)}$, with
$$V_n^{(R)}=\left(\vep^{(R)}_{2n},\vep^{(R)}_{2n+1}\right)^t \qquad {\rm and}\qquad  Y_{k,n}^{(r)}=\Big(1,\xx_{k,n}^{(r)}\Big)^t,$$
where
\[
\vep_k^{(R)}=\vep_k{\bf 1}_{\{|\vep_k|\le R\}}-\ee\left[\vep_k{\bf
1}_{\{|\vep_k|\le R\}}\right], \qquad \xx_{k,n}^{(r)}=\xx_k{\bf
1}_{\Big\{||\xx_{k}||\le
r\frac{\sqrt{|\ttt_{n-1}|}}{b_{|\ttt_{n-1}|}}\Big\}}.
\]

We introduce $\Gamma^{(R)}$ the conditional covariance matrix associated with
$(\epsilon^{(R)}_{2k},\epsilon^{(R)}_{2k+1})^t$ and the truncated matrix associated with $S_n$ :
\begin{equation*}
\Gamma^{(R)}=\begin{pmatrix}
\sigma_R^2 & \rho_R\\
\rho_R & \sigma_R^2
\end{pmatrix}
 \qquad {\rm and} \qquad S_n^{(r)} =\sum_{k\in \ttt_{n,p-1}}\begin{pmatrix}
1& (\xx^{(r)}_{k,n})^t\\
\xx^{(r)}_{k,n} & \xx^{(r)}_{k,n}(\xx^{(r)}_{k,n})^t\\
\end{pmatrix}.
\end{equation*}

The condition {\bf (P2)} in Puhalskii's Theorem \ref{mdppuh} is
verified by the construction of the truncated martingale, that is
for some positive constant $c$, we have that for all $k\in
\ttt_{n-1}$
$$||D_{k,n}^{(r,R)}||\le c\frac{ \sqrt {|\ttt_{n-1}|} } {b_{|\ttt_{n-1}|}}.$$

From Proposition \ref{lemma:lyapounovexpo}, we also have for all $r>0$,
\begin{equation}
\label{Linder} \frac{1}{|\ttt_{n-1}|}\sum_{k\in\ttt_{n-1,p-1}}
\xx_{k} \mathrm{I}_{\left\{||\xx_{k}|| > r\frac{\sqrt
{|\ttt_{n-1}|}}{b_{|\ttt_{n-1}|}} \right\}} \superexp 0;
\end{equation}
and
\begin{equation}
\label{Linder2}
\frac{1}{|\ttt_{n-1}|}\sum_{k\in\ttt_{n-1,p-1}}\xx_{k}\xx_{k}^t
\mathrm{I}_{\left\{||\xx_{k}|| > r\frac{\sqrt
{|\ttt_{n-1}|}}{b_{|\ttt_{n-1}|}} \right\}} \superexp 0.
\end{equation}
From (\ref{Linder}) and (\ref{Linder2}), we deduce that for all $r>0$
\begin{equation}\label{S_ntrunc} \frac{1}{|\ttt_{n-1}|}\left(S_{n-1}-S_{n-1}^{(r)} \right)\superexp 0.\end{equation}

Then, we easily transfer the properties $\eqref{croch}$  to the
truncated martingale $(M_n^{(r,R)})_{n\ge 0}$. We have for all $R>0$
and all $r>0$,
\begin{equation*}
\frac{\langle M^{(r,R)} \rangle_n}{|\ttt_{n-1}|} =
\Gamma^{(R)}\otimes\frac{S_{n-1}^{(r)}}{|\ttt_{n-1}|}=
-\Gamma^{(R)}\otimes\left( \frac{S_{n-1}-S_{n-1}^{(r)}
}{|\ttt_{n-1}|}\right)+\Gamma^{(R)}\otimes\frac{S_{n-1}}{|\ttt_{n-1}|}
\superexp \Gamma^{(R)} \otimes L
\end{equation*}
That is condition {\bf (P1)} in Puhalskii's Theorem \ref{mdppuh}.

Note also that Proposition \ref{lemma:lyapounovexpo} work for the
truncated martingale $(M_n^{(r,R)})_{n\ge 0}$, which ensures the
Lindeberg's condition and thus condition {\bf (P3)} to
$(M_n^{(r,R)})_{n\ge 0}$. By Theorem \ref{mdppuh} in the Appendix,
we deduce that
$(M_n^{(r,R)}/(b_{|\ttt_{n-1}|}\sqrt{|\ttt_{n-1}|}))_{n\ge 0}$
satisfies a MDP on $\dR^{2(p+1)}$ with speed $b_{|\ttt_{n-1}|}^2$
and good rate function given by

\begin{equation}\label{rateIR}
I_R(x)=\frac 12 x^t  (\Gamma^{(R)}\otimes L)^{-1} x
.\end{equation}

\bigskip
{\bf Step 2.} At first, we infer from the hypothesis  {\bf (Ea)} that:
\begin{enumerate}
\item[{\bf (N1R)}] there is a sequence $(\kappa_{R})_{R>0}$ with $\kappa_{R}\longrightarrow 0$
when $R$ goes to infinity, such that for all $n\geq
p-1,$ for all $k\in \g_{n+1}$, for all $t\in \rr$ and for $R$ large
enough
\[
\ee\left[\exp t \left(\vep_{k} - \vep_{k}^{R}\right)/\FF_{n}\right]\leq
\exp\left(\frac{\kappa_{R} t^{2}}{2}\right),\,\,\, a.s.
\]
\end{enumerate}

 The approximation, in the sense of the moderate
deviation, is described by the following convergence, for all $r> 0$
and all $\delta>0$,
\[
\limsup_{R\rightarrow\infty}\limsup_{n\rightarrow\infty}
\frac{1}{b_{|\ttt_{n-1}|}^{2}} \log\pp\left(
\frac{\|M_n-M_n^{(r,R)}\|}{\sqrt{|\ttt_{n-1}|}b_{|\ttt_{n-1}|}}>\delta\right)
= -\infty.
\]
For that, we shall prove that for $\eta\in\{0,1\}$
\begin{equation}\label{negI1}
I_1=\frac{1}{\sqrt{|\ttt_{n-1}|}b_{|\ttt_{n-1}|}} \sum_{k\in
\ttt_{n-1,p-1}}\left(\vep_{2k+\eta}-\vep^{(R)}_{2k+\eta}\right) \superexp 0,
\end{equation}
\begin{equation}\label{negI2}
I_2=\frac{1}{\sqrt{|\ttt_{n-1}|}b_{|\ttt_{n-1}|}} \sum_{k\in
\ttt_{n-1,p-1}}\left(\vep_{2k+\eta}\xx_k-\vep^{(R)}_{2k+\eta}\xx^{(r)}_{k,n}\right)\superexp
0.
\end{equation}
To prove (\ref{negI1}) and (\ref{negI2}), we have to do it only for $\eta=0$ the same proof works for $\eta=1$.

\bigskip
{\bf Proof of (\ref{negI1}) }
We have  for all $\alpha >0$ and $R$ large enough
\begin{eqnarray*}
&\,&\ee\left(\exp\left(\alpha \sum_{k\in\ttt_{n-1,p-1}}(\vep_{2k}-\vep^{(R)}_{2k})\right)\right)\\
&=&\ee\left[\prod_{k\in\ttt_{n-2,p-1}}\exp\left(\alpha(\vep_{2k}-\vep^{(R)}_{2k})\right)\times \ee\left[\prod_{k\in\g_{n-1}}\exp\left(\alpha(\vep_{2k}-\vep^{(R)}_{2k})\right)\Big/\FF_{n-1}\right]\right]\\
&=& \ee\left[\prod_{k\in\ttt_{n-2,p-1}}\exp\left(\alpha(\vep_{2k}-\vep^{(R)}_{2k})\right)\times\prod_{k\in\g_{n-1}}\ee\left[\exp\left(\alpha(\vep_{2k}-
\vep^{(R)}_{2k})\right)\Big/\FF_{n-1}\right]\right] \\
&\le& \ee\left[\prod_{k\in\ttt_{n-2,p-1}}\exp\left(\alpha(\vep_{2k}-\vep^{(R)}_{2k})\right)\exp\left(|\g_{n-1}|
\alpha^2\kappa_R\right)\right]\\
&\le& \exp\left(|\ttt_{n-1}|\alpha^2\kappa_R\right).
\end{eqnarray*}

where hypothesis {\bf (N1R)} was used to get the first inequality,
and the second was obtained by induction. By Chebyshev inequality and  the
previous calculation applied to $\alpha=\lambda b_{|\ttt_{n-1}|}/|\ttt_{n-1}|$, we obtain for all $\delta>0$

$$\pp\left(\frac{1}{\sqrt{|\ttt_{n-1}|}b_{|\ttt_{n-1}|}} \sum_{k\in
\ttt_{n-1,p-1}}(\vep_{2k}-\vep^{(R)}_{2k})\ge \delta\right)\le
\exp\left(-b^2_{|\ttt_{n-1}|} \Big(\delta\lambda-\kappa_R\lambda^2\Big)\right).$$

Optimizing on $\lambda$, we obtain
\[
\frac{1}{b^2_{|\ttt_{n-1}|}}
\log\pp\left(\frac{1}{\sqrt{|\ttt_{n-1}|}b_{|\ttt_{n-1}|}}
\sum_{k\in \ttt_{n-1,p-1}}\left(\vep_{2k}- \vep^{(R)}_{2k}\right)\ge
\delta\right)\le -\frac{\delta^2}{4\kappa_R}.
\]
Letting $n$ goes to infinity  and than $R$ goes to infinity, we
obtain the negligibility in (\ref{negI1}).
\bigskip

{\bf Proof of (\ref{negI2}) }
Now, since we have the decomposition
\[
\vep_{2k}\xx_k-\vep^{(R)}_{2k}\xx^{(r)}_{k,n}=\left(\vep_{2k}-\vep^{(R)}_{2k}\right)\xx_{k,n}^{(r)}+
\vep_{2k}\left(\xx_k-\xx^{(r)}_{k,n}\right),
\]
we introduce the following notations
\[
L_n^{(r)}= \sum_{k\in
\ttt_{n-1,p-1}}\vep_{2k}\left(\xx_k-\xx^{(r)}_{k,n}\right)\qquad{\rm and}\qquad F_n^{(r,R)}= \sum_{k\in
\ttt_{n-1,p-1}}\left(\vep_{2k}-\vep^{(R)}_{2k}\right)\xx^{(r)}_{k,n}.
\]

To prove (\ref{negI2}), we will show that for all $r>0$
\begin{equation}\label{negL}
\frac{L_n^{(r)}}{\sqrt{|\ttt_{n-1}|}b_{|\ttt_{n-1}|}} \superexp
0,
\end{equation}
and for all $r>0$ and all $\delta>0$
\begin{equation}\label{negF}
\limsup\limits_{R\rightarrow \infty} \limsup\limits_{n\rightarrow
\infty} \frac{1}{b_{|\ttt_{n-1}|}^{2}} \log\pp\left(
\frac{\|F_{n}^{(r,R)}\|}{b_{|\ttt_{n-1}|}\sqrt{|\ttt_{n-1}|}}
> \delta\right) = - \infty.
\end{equation}

Let us first deal with $(L^{(r)}_n)$. Let its first component
\[
L_{n,1}^{(r)} =
\sum\limits_{k\in\ttt_{n-1,p-1}}\vep_{2k}\left(X_{k}-X^{(r)}_{k,n}\right).
\]

For $\lambda\in \rr$, we consider the random sequence
$(Z^{(r)}_{n,1})_{n\geq {p-1}}$ defined by
\[
Z^{(r)}_{n,1}=\exp\left(\lambda L_{n,1}^{(r)}- \frac{\lambda^{2}\phi}{2}
\sum\limits_{k\in\ttt_{n-1,p-1}}X_{k}^{2}\mathbf{1}_{\left\{
\|\xx_{k}\|>r\frac{ \sqrt{|\ttt_{n-1}|}}{b_{|\ttt_{n-1}|}} \right\}} \right)
\]
where $\phi$ appears in {\bf (N1)}.

For $b>0$, we introduce the following event
\[
A_{n,1}^{(r)}(b) = \left\{\frac{1}{|\ttt_{n-1}|}\sum\limits_{k\in\ttt_{n-1,p-1}}X_{k}^{2} \mathbf{1}_{\left\{\|\xx_{k}\|>r\frac{ \sqrt{|\ttt_{n-1}|}}{b_{|\ttt_{n-1}|}} \right\}}>b\right\}.
\]

Using {\bf (N1)},  we have for all $\delta>0$
\begin{eqnarray}\label{surmarZ}
&\,&\pp\left(\frac{1}{b_{|\ttt_{n-1}|\sqrt{|\ttt_{n-1}|}}} L_{n,1}^{(r)}>\delta\right)\notag
\\ &\le& \pp\Big(A_{n,1}^{(r)}(b)\Big) + \pp\left(Z_{n,1}^{(r)}> \exp\left(\delta\lambda b_{|\ttt_{n-1}|}\sqrt{|\ttt_{n-1}|} -
\frac{\lambda^{2}\phi}{2}b |\ttt_{n-1}|\right)\right)\notag \\ &\leq& \pp\Big(A_{n,1}^{(r)}(b)\Big) +
\exp\left(-b_{|\ttt_{n-1}|}\sqrt{|\ttt_{n-1}|} \left(\delta\lambda -
\frac{b\phi\sqrt{|\ttt_{n-1}|}}{2b_{|\ttt_{n-1}|}}\lambda^{2}\right)\right),
\end{eqnarray}
where the second term in (\ref{surmarZ}) is obtained by conditioning
successively on $(\GG_i)_{2^{p-1}\le i\le |\ttt_{n-1}|-1}$ and using
the fact that
$$\ee\left[\exp\left(\lambda\vep_{2^p}\left(X_{2^{p-1}}-X^{(r)}_{2^{p-1}}\right)-\frac{\lambda^2\phi}{2}X_{2^{p-1}}^{2}\mathbf{1}_{\Big\{\|\xx_{2^{p-1}}\|>r  \frac{ \sqrt{  2^{p-1}  }  }  {  b_{2^{p-1}} } \Big\}}\right)\right]\le 1,$$
which follows from {\bf (N1)}.

From Proposition \ref{lemma:lyapounovexpo}, we have for all $b>0$
\[
\limsup_{n\rightarrow \infty} \frac{1}{b_{|\ttt_{n-1}|}^{2}} \log\pp\Big(A_{n,1}^{(r)}(b)\Big) = -\infty,
\]
so that taking $\lambda = \delta b_{|\ttt_{n-1}|} / (b\phi\sqrt{|\ttt_{n-1}|})$ in (\ref{surmarZ}), we are led to
\[
\limsup\limits_{n\rightarrow \infty} \frac{1}{b_{|\ttt_{n-1}|}^{2}}
\log\pp\left(
\frac{L_{n,1}^{(r)}}{b_{|\ttt_{n-1}|}\sqrt{|\ttt_{n-1}|}}
> \delta\right) \leq - \frac{\delta^{2}}{2b\phi}.
\]
Letting $b\rightarrow 0$, we obtain that the right hand of the last inequality goes to $-\infty$.
Proceeding in the same way for $-L_{n,1}^{(r)},$ we deduce that for all $r>0$
\[
\frac{L_{n,1}^{(r)}}{b_{|\ttt_{n-1}|}\sqrt{|\ttt_{n-1}|}}\superexp 0.
\]
Now,  it is easy to check that the same proof works for the others
components of $L_{n}^{(r)}.$ We thus conclude the proof of (\ref{negL}).

Eventually, let us treat the term $(F_{n}^{(r,R)}).$ We follow the same
approach as in the proof of (\ref{negL}). Let its first component
\[
F_{n,1}^{(r,R)} = \sum\limits_{k\in \ttt_{n-1,p-1}} (\vep_{2k}-\vep_{2k}^{(R)}) X_{k,n}^{(r)}
\]

For $\lambda\in \rr$, we consider the random sequence
$\left(W_{n,1}^{(r,R)}\right)_{n\geq {p-1}}$ defined by
\[
W_{n,1}^{(r,R)}=\exp\left(\lambda\sum_{k\in\ttt_{n-1,p-1}}(\vep_{2k}-\vep^{(R)}_{2k})X_{k,n}^{(r)}-
\frac{\lambda^2\kappa_R}{2}\sum\limits_{k\in\ttt_{n-1,p-1}}(X_{k,n}^{(r)})^{2}\right)
\]
where $\kappa_R$ appears in {\bf (N1R)}.

Let $b>0.$ Consider the following event $B_{n,1}^{(r)}(b) = \left\{\frac{1}{|\ttt_{n-1}|}\sum\limits_{k\in
\ttt_{n-1,p-1}}(X_{k,n}^{(r)})^{2}>b\right\}.$

We have for all $\delta>0,$
\begin{eqnarray}\label{surmarW}
&\,&\pp\left(\frac{F_{n,1}^{(r,R)}}{b_{|\ttt_{n-1}|}\sqrt{|\ttt_{n-1}|}}
> \delta\right)\notag\\
&\leq& \pp\left(B_{n,1}^{(r)}(b)\right) + \pp\left(W_{n,1}^{(r, R)} > \exp\left(\delta
\lambda b_{|\ttt_{n-1}|} \sqrt{|\ttt_{n-1}|} -\frac{\lambda^2\kappa_R}{2}|\ttt_{n-1}|b\right)\right)\notag\\
&\leq& \pp\left(B_{n,1}^{(r)}(b)\right) + \exp\left(-b_{|\ttt_{n-1}|} \sqrt{|\ttt_{n-1}|} \left(\delta
\lambda -\frac{b\kappa_R\sqrt{|\ttt_{n-1}|} }{2b_{|\ttt_{n-1}|}}\lambda^2\right)\right)
\end{eqnarray}
where the second term in (\ref{surmarW}) is obtained by conditioning
successively on $(\GG_i)_{2^{p-1}\le i\le |\ttt_{n-1}|-1}$ and using
the fact that
$$\ee\left[ \exp\left (\lambda \left(\vep_{2^p}-\vep^{(R)}_{2^p}\right)X^{(r)}_{2^{p-1}}-\frac{\lambda^2\kappa_R}{2}\left(X^{(r)}_{2^{p-1}}\right)^2\right)\right]\le 1,$$

Since  $B_{n,1}^{(r)}(b) \subset\left\{\frac{1}{|\ttt_{n-1}|}\sum\limits_{k\in
\ttt_{n-1,p-1}}X_{k}^{2}>b\right\} ,$ from Proposition \ref{convexpocrochet2}, we deduce that for $b$ large enough
\[
\limsup\limits_{n\rightarrow\infty} \frac{1}{b_{|\ttt_{n-1}|}^{2}}
\log\pp\left( B^{(r)}_{n,1}(b) \right) = -\infty,
\]
so that choosing $\lambda = \delta b_{|\ttt_{n-1}|} / (\kappa_Rb\sqrt{|\ttt_{n-1}|}),$ we get for all $\delta >0$
\[
\limsup\limits_{n\rightarrow \infty}  \frac{1}{b_{|\ttt_{n-1}|}^{2}}\log\pp\left(
\frac{F_{n,1}^{(r,R)}}{b_{|\ttt_{n-1}|}\sqrt{|\ttt_{n-1}|}} >
\delta\right) \leq -\frac{\delta^{2}}{2\kappa_Rb}.
\]
Letting $R$ to infinity, we obtain that
\[
\limsup\limits_{R\rightarrow \infty} \limsup\limits_{n\rightarrow \infty}  \frac{1}{b_{|\ttt_{n-1}|}^{2}}\log\pp\left(
\frac{F_{n,1}^{(r,R)}}{b_{|\ttt_{n-1}|}\sqrt{|\ttt_{n-1}|}} >
\delta\right) = -\infty.
\]
Now it is easy to check that the same works for $-F_{n,1}^{(r,R)}$ and
for the others components of $F_{n}^{(r,R)}.$ We thus conclude (\ref{negF})
for all $r>0$.

\bigskip

{\bf Step 3.} By application of  Theorem 4.2.16 in \cite{DemZei98}, we find
that $(M_{n}/(b_{|\ttt_{n-1}|}\sqrt{|\ttt_{n-1}|}))$ satisfies an MDP on
$\dR^{2(p+1)}$ with speed $b_{|\ttt_{n-1}|}^2$ and rate function
\begin{equation*}
\widetilde{I}(x) = \sup_{\delta >0}\liminf_{R\rightarrow \infty}\inf_{z\in B_{x,\delta}}I_R(z),
\end{equation*}
where $I_R$ is given in \eqref{rateIR}  and $B_{x,\delta}$ denotes
the ball $\{z:|z-x|<\delta\}.$ The identification of the rate
function $\widetilde I= I_M$, where $I_M$ is given in \eqref{IM} is
done easily (see for example \cite{DjGu01}), which concludes the
proof of Theorem \ref{thm:mdp_M_n}.

\bigskip
{\bf Proof in the case 1.}

For the proof in the case 1, there are no
change in Step 1, and Step 3, instead of (\ref{Linder}),
(\ref{Linder2}), and {\bf(N1)}, we use Remark \ref{Markov} and {\bf
(G1)}. In Step 2, the negligibility in (\ref{negI1}), comes from the
MDP of the i.i.d. sequences $(\vep_{2k}-\vep^{(R)}_{2k})$ since it
verifies the condition, for $\lambda>0$ and all $R>0$
$$\ee(\exp(\lambda (\vep_{2k}-\vep^{(R)}_{2k} ))<\infty.$$
The negligibility of $(L_n^{(r)})$ works in the same way. For
$(F_n^{(r,R)})$ we will use the MDP for martingale,  see Proposition
\ref{mdpdj}. For $R$ large enough, we have
\begin{eqnarray*}
\dP\Bigg(\left\vert X_{k,n}^{(r)} \left(\vep_{2k} -\vep_{2k}^{(R)} \right) \right\vert
> b_{|\ttt_{n-1}|}\sqrt{|\ttt_{n-1}|} ~ \Big\vert \cF_{k-1}\Bigg) & \leq & \dP\left(\left\vert \vep_{2k} -
\vep_{2k}^{(R)} \right\vert > \frac{b_{|\ttt_{n-1}|}^{2}}{r} \right),\\
 & = & \dP\left( \left\vert \vep_{2} -
\vep_{2}^{(R)} \right\vert > \frac{b_{|\ttt_{n-1}|}^{2}}{r}\right) = 0.
\end{eqnarray*}
This implies that
\begin{equation*}
\limsup_{n\rightarrow \infty} \frac{1}{b_{|\ttt_{n-1}}^{2}|} \log \left( |\ttt_{n-1}|\,\,\,
\underset{k \geq 1} {\rm ess\,sup} ~
\dP\Bigg(\left\vert X_{k,n}^{(r)} \left(\vep_{2k} - \vep_{2k}^{(R)} \right) \right\vert
> b_{|\ttt_{n-1}|}\sqrt{|\ttt_{n-1}|} ~ \Big\vert \cF_{k-1}\Bigg) \right) = -\infty.
\end{equation*}
That is condition {\bf (D2)} in Proposition \ref{mdpdj}.

For all $\gamma > 0$ and all $\delta > 0$, we obtain from Remark \ref{Markov} , that
\begin{eqnarray*}
\limsup_{n\rightarrow \infty} \frac{1}{b_{|\ttt_{n-1}|}^{2}}
\log \dP\left( \frac{1}{|\ttt_{n-1}|} \sum_{k\in \ttt_{n-1,p-1}} \left( X_{k,n}^{(r)} \right)^{2} \mathrm{I}_{\left\{ \vert X_{k,n}^{(r)} \vert > \gamma \frac{\sqrt{|\ttt_{n-1}|}}{b_{|\ttt_{n-1}|}}\right\}} > \delta\right)  \\
\leq \limsup_{n\rightarrow \infty} \frac{1}{b_{|\ttt_{n-1}|}^{2}} \log \dP\left( \frac{1}{|\ttt_{n-1}|} \sum_{k\in \ttt_{n-1,p-1}} X_{k}^{2} {\bf 1}_{\left\{\vert X_{k} \vert > \gamma \frac{\sqrt{|\ttt_{n-1}|}}{b_{|\ttt_{n-1}|}}\right\}} > \delta\right) = -\infty.
\end{eqnarray*}
That is condition {\bf (D3)} in Proposition \ref{mdpdj}.
Finally, from Remark \ref{Markov} and in the same way as in (\ref{S_ntrunc}), it follows that
\begin{equation*}
\frac{\langle F^{(r,R)} \rangle_{n,1}}{|\ttt_{n-1}|} =  Q_R\frac{1}{|\ttt_{n-1}|} \sum_{k\in\ttt_{n-1,p-1} }( X_{k,n}^{(r)} )^2 \superexp Q_R  \ell
\end{equation*}
for some positive constant $\ell$, where
$Q_{R}=\ee\left[\left(\vep_{2} - \vep_{2}^{(R)}\right)^{2}\right].$
That is condition {\bf (D1)} in Proposition \ref{mdpdj}. Moreover,
it is clear that $Q_{R}$ converges to 0 as $R$ goes to infinity. In
light of foregoing, we infer from Proposition \ref{mdpdj}, that
$(F_{n,1}^{(r,R)}/(b_{|\ttt_{n-1}|}\sqrt{|\ttt_{n-1}}|))$ satisfies
an MDP on $\dR$ of speed $b_{|\ttt_{n-1}|}^{2}$ and rate
function $I_{R}(x) = x^{2}/(2 Q_{R} \ell).$
In particular, this implies that for all $\delta > 0$,
$$
\limsup\limits_{n\rightarrow\infty} \frac{1}{b_{|\ttt_{n-1}|}^{2}}
\log\dP\left( \frac{|F_{n,1}^{(r,R)}|}{b_{|\ttt_{n-1}|}\sqrt{|\ttt_{n-1}|}}>\delta\right) \leq
-\frac{\delta^{2}}{2Q_{R}\ell},
$$
and letting $R$ go to infinity clearly leads to the result.

\subsubsection{Proof of Theorem \ref{thm:mdp_theta_n}} The proof works in the case 1 and in the case 2. From (\ref{thetaM_n}), we have
$$\frac{\sqrt{|\ttt_{n-1}|}}{b_{|\ttt_{n-1}|}}(\hat\theta_n-\theta)=|\ttt_{n-1}|\Sigma^{-1}_{n-1}\frac{M_n}{b_{|\ttt_{n-1}|}|\ttt_{n-1}|}$$
From Proposition \ref{prop:convergence_crochet}, we obtain that
\begin{equation}\label{cro}\frac{\Sigma_n}{|\ttt_{n}|}=I_2\otimes \frac{S_n}{|\ttt_{n}|}\superexpn I_2\otimes L. \end{equation}
According to Lemma 4.1 of  \cite{Wor01c}, together with (\ref{cro}), we deduce that
\begin{equation}\label{cro1}|\ttt_{n-1}|\Sigma^{-1}_{n-1}\superexp I_2\otimes L^{-1}. \end{equation}

From Theorem \ref{thm:mdp_M_n}, (\ref{cro1}) and the contraction
principle \cite{DemZei98}, we deduce that the sequence
$\big(\sqrt{|\ttt_{n-1}|}(\hat\theta_n-\theta)/b_{|\ttt_{n-1}|}\big)_{n\ge
1}$ satisfies the MDP with rate function $I_{\theta}$ given by
(\ref{Itheta}).

\end{proof}


\subsection{Proof of Theorem \ref{mdp_sigma_n_rho_n}}$\,$
\medskip

For the proof of Theorem \ref{mdp_sigma_n_rho_n}, the  case 1 is an
easy consequence of the classical MDP for i.i.d.r.v. applied to the
sequence ($\vep_{2k}^{2} + \vep_{2k+1}^{2}$) , for the case 2,  we
will use Proposition \ref{mdpdj}, rather than Puhalskii's
Theorem \ref{mdppuh}.
\medskip

We will prove that the
sequence$\displaystyle\left(\sqrt{|\ttt_{n-1}|}(\sigma_{n}^{2} -
\sigma^{2})/b_{|\ttt_{n-1}|}\right)$ satisfies the MDP. For that,we
will prove that conditions {\bf (D1), (D2)} and  {\bf (D3)} of Proposition \ref{mdpdj}  are verified. Let us consider
the $\GG_{n}$-martingale $(N_{n})_{n\geq 2^{p-1}}$ given by
\[
N_{n} = \sum\limits_{k=2^{p-1}}^{n} \nu_{k}, \hspace{0.25cm}
\text{where $\nu_{k} = \vep_{2k}^{2} + \vep_{2k+1}^{2} -
2\sigma^{2}.$}
\]
It is easy to see that its predictable quadratic variation is given
by
\[
\langle N\rangle_{n} = \sum\limits_{k=2^{p-1}}^{n}
\ee\left[\nu_{k}^{2}/\GG_{k-1}\right] =
(n-2^{p-1}+1)(2\tau^{4}-4\sigma^{4}+2\nu^{2}),
\]
which immediately implies that
\begin{equation*}
\frac{\langle N\rangle_{n}}{n} \superexpp 2\tau^{4}-4\sigma^{4}+2\nu^{2},
\end{equation*}
ensuring condition  {\bf (D1)} in Proposition \ref{mdpdj}.

Next, for $B>0$ large enough, we have for $a>2$ (in {\bf (Ea)}), and some positive constant $c$
\begin{equation*}
\pp\left(\frac{1}{n}
\sum\limits_{k=2^{p-1}}^{n} |\nu_{k}|^{a} > B\right)\leq 3
\max_{\eta\in\{0,1\}}\left\{ \pp\left(\frac{1}{n} \sum\limits_{k=2^{p-1}}^{n}
|\vep_{2k+\eta}|^{2a} > \frac{B}{3c}\right)\right\}.
\end{equation*}

From hypothesis {\bf (Ea)} and since $B$ is large enough, we obtain,
for a suitable $t>0$ via the Chernoff inequality and several
successive conditioning on $(\GG_n)$ , for $\eta\in\{0,1\}$
\[
\pp\left(\frac{1}{n} \sum\limits_{k=2^{p-1}}^{n} |\vep_{2k+\eta}|^{2a} >
\frac{B}{3c}\right) \leq \exp\left(-tn\left(\frac{B}{3c} - \log
E\right)\right) \leq \exp\left(-tc'n\right),
\]
where $c$, $c'$ are a positive generic constant.
Therefore, for $B>0$ large enough, we deduce that
\[
\limsup_{n\rightarrow \infty}\frac{1}{n}\log\pp\left(\frac{1}{n}
\sum\limits_{k=2^{p-1}}^{n} |\nu_{k}|^{a} > B\right) < 0,
\]
and this implies (see e.g \cite{Wor01c}) exponential Lindeberg
condition, that is for all $r>0$

\begin{equation*}
\frac{1}{n}\sum\limits_{k=2^{p-1}}^{n}
\nu_{k}^{2} \mathbf{1}_{\left\{|\nu_{k}| > r
\frac{\sqrt{n}}{b_{n}}\right\}}     \superexpp 0.
\end{equation*}
That is condition {\bf(D3)}  in Proposition \ref{mdpdj}.

Now, for all $k\in \nn$ and a suitable $t>0$ we have
\begin{eqnarray*}
\pp\left(|\nu_{k}|>b_{n}\sqrt{n}/\GG_{k-1}\right) &\leq& \sum_{\eta=0}^1\pp\left(
|\vep_{2k+\eta}^{2} - \sigma^{2}| >\frac{b_{n}\sqrt{n}}{2}/\GG_{k-1}\right) \\
&\leq&
\exp\left(\frac{-tb_{n}\sqrt{n}}{2}\right)
\sum_{\eta=0}^1\ee\Big[\exp\left(t|\vep_{2k+\eta}^{2} -
\sigma^{2}|\right)/\GG_{k-1}\Big]
\\ &\leq& 2E'\exp\left(\frac{-tb_{n}\sqrt{n}}{2}\right),
\end{eqnarray*}
where from hypothesis {\bf (Na)}, $E'$ is finite and positive. We
are thus led to
\[
\frac{1}{b_{n}^{2}} \log\left(n\,\,\,\underset {k\in \nn^{*}}{\rm
ess\,sup} \pp\left(|\nu_{k}|> b_{n}\sqrt{n}/\GG_{k-1}\right)\right)
\leq \frac{\log(2E'n)}{b_{n}^{2}} - \frac{t\sqrt{n}}{b_{n}},
\]
and consequently, letting $n$ goes to infinity, we get the condition {\bf(D2)}  in Proposition \ref{mdpdj}.

Now, applying Proposition  \ref{mdpdj}, we conclude
that $(N_{n}/(b_{n}\sqrt{n}))_{n\ge 0}$
satisfies the MDP with speed $b_{n}^{2}$
and rate function
\[
I_{N}(x) = \frac{x^{2}}{4(\tau^{4} - 2\sigma^{4} + 2\nu^{2})}.
\]
Applying the foregoing to $|\ttt_{n-1}|$ and using contraction
principle (see e.g \cite{DemZei98}), we deduce that the sequence
\[
\frac{\sqrt{|\ttt_{n-1}|}}{b_{|\ttt_{n-1}|}}(\sigma_{n}^{2} -
\sigma^{2}) =
\frac{N_{|\ttt_{n-1}|}}{2b_{|\ttt_{n-1}|}\sqrt{|\ttt_{n-1}|}}
\]
satisfies a MDP with speed
$b_{|\ttt_{n-1}|}^{2}$ and rate function $I_{\sigma^2}$ given by (\ref{Isigma}).


\bigskip
We obtain as in the proof of  the first part, with a slight modification that the sequence
$(|\ttt_{n-1}| (\rho_{n} - \rho)/b_{|\ttt_{n-1}|})$ satisfies a MDP with speed
$b_{|\ttt_{n-1}|}^{2}$ and rate function $I_{\rho}$ given by (\ref{Irho}).

\subsection{Proof of Theorem \ref{thm:conv_expo_sigma_n_rho_n}} Here also the proof works for the two cases.

\medskip
Let us first deal with $\hat{\sigma}_{n}.$ We have
\begin{equation*}
\hat{\sigma}_{n}^{2} - \sigma^{2} = (\hat{\sigma}_{n}^{2} -
\sigma_{n}^{2}) + (\sigma_{n}^{2} - \sigma^{2}).
\end{equation*}
From (\ref{devineqeps2}) and (\ref{devineqeps3}), we easily deduce
that $\sigma_{n}^{2} \superexp \sigma^{2}$ in the case 1 and in the
case 2. Thus, it is enough to prove that $\hat{\sigma_{n}^{2}} -
\sigma_{n}^{2} \superexp 0.$ Let $\theta^{(0)} = (a_{0}, a_{1},
\cdots, a_{p})^{t},$ $\theta^{(1)} = (b_{0}, b_{1}, \cdots,
b_{p})^{t},$ $\hat{\theta}_{n}^{(0)} = (\hat{a}_{0,n},
\hat{a}_{1,n}, \cdots, \hat{a}_{p,n}),$ $\hat{\theta}_{n}^{(1)} =
(\hat{b}_{0,n}, \hat{b}_{1,n}, \cdots, \hat{b}_{p,n}).$

Let us introduce the following function $f$ defined for $x$ and $z$
in $\dR^{p+1}$ by
\[
f(x,z) = \left(x_{1} - z_{1} - \sum\limits_{i=2}^{p+1} z_{i}
x_{i}\right)^{2},
\]
where $x_{i}$ and $z_{i}$ denote respectively the $i$-th component of
$x$ and $z.$ One can observe that
\begin{align*}
\hat{\sigma}_{n}^{2} - \sigma_{n}^{2} &= \frac{1}{2|\ttt_{n-1}|}
\sum\limits_{k\in \ttt_{n-1,p-1}} \left\{f\left(\xx_{2k},
\hat{\theta}_{n}^{(0)}\right) - f\left(\xx_{2k},
\theta^{(0)}\right)\right\} \\ &+ \frac{1}{2|\ttt_{n-1}|}
\sum\limits_{k\in \ttt_{n-1,p-1}} \left\{f\left(\xx_{2k+1},
\hat{\theta}_{n}^{(1)}\right) - f\left(\xx_{2k+1},
\theta^{(1)}\right)\right\}.
\end{align*}
By the Taylor-Lagrange formula, $\forall x\in \dR^{p+1}$ and
$\forall z, z'\in\dR^{p+1}$, one can find $\lambda \in (0,1)$ such
that
\begin{equation*}
f\left(x,z'\right) - f\left(x,z\right) = \sum\limits_{j=1}^{p+1}
(z'_{j} - z_{j})\partial_{z_{j}} f\left(x, z + \lambda (z' -z)\right).
\end{equation*}
Let the function $g$ defined by
\begin{equation*}
g(x,z) = x_{1} - z_{1} - \sum\limits_{j=2}^{p+1} z_{j}x_{j}.
\end{equation*}
Observing that
\begin{equation*}
\begin{cases}
\frac{\partial f}{\partial z_{1}}(x,z) = -2g(x,z) \\
\frac{\partial f}{\partial z_{j}}(x,z) = -2 x_{j} g(x,z) \quad
\forall j\geq 2,
\end{cases}
\end{equation*}
we get easily that $\left|\frac{\partial f}{\partial
z_{j}}(x,z)\right| \leq 4(1 + \|z\|)(1 + \|x\|^{2})$ for all
$j\geq 1,$ and this implies
\begin{equation*}
\left|f(x,z') - f(x,z)\right| \leq c \|z'-z\| \left(1 + \|z\| +
\|z' - z\|\right) \left(1 + \|x\|^{2}\right),
\end{equation*}
for some positive constant $c.$ Now, applying the foregoing to
$f\left(\xx_{2k}, \hat{\theta}_{n}^{(0)}\right) - f\Big(\xx_{2k},
\theta^{(0)}\Big)$ and to $f\left(\xx_{2k+1},
\hat{\theta}_{n}^{(1)}\right) - f\Big(\xx_{2k+1},
\theta^{(1)}\Big),$ we deduce easily that
\begin{equation*}
|\hat{\sigma}_{n}^{2} - \sigma_{n}^{2}| \leq c\|\hat{\theta}_{n} -
\theta\|\left(1 + \|\theta\| + \|\hat{\theta}_{n} - \theta\|\right)
\frac{1}{|\ttt_{n-1}|} \sum\limits_{k\in \ttt_{n-1,p-1}} \left(1 +
\|\xx_{k}\|^{2}\right),
\end{equation*}
for some positive constant $c$. From the MDP of $\hat{\theta}_{n} - \theta,$ we infer that
\begin{equation}\label{conv_superexpo_theta}
\|\hat{\theta}_{n} - \theta\| \superexp 0.
\end{equation}
Form Proposition \ref{convexpocrochet2} we deduce that
\begin{equation}\label{conv_superexpo_trace_S}
\frac{1}{|\ttt_{n-1}|} \sum\limits_{k\in \ttt_{n-1,p-1}} \left(1 +
\|\xx_{k}\|^{2}\right) \superexp 1+ {\rm Tr}(\Lambda).
\end{equation}
We thus conclude via (\ref{conv_superexpo_theta}) and
(\ref{conv_superexpo_trace_S}) that
\begin{equation*}
\hat{\sigma_{n}^{2}} - \sigma_{n}^{2} \superexp 0.
\end{equation*}
This ends the proof for $\hat{\sigma}_{n}.$ The proof for
$\hat{\rho}_{n}$ is very similar and uses hypothesis {\bf (G2')} and
{\bf (N2')} to get inequalities similar to (\ref{devineqeps2}) and
(\ref{devineqeps3}).

\vspace{15pt}

{\bf Acknowledgments.}  The authors  thank Arnaud Guillin for all his advices and suggestions during the preparation of this work.


\vspace{5pt}




\begin{thebibliography}{10}

\bibitem{BZ4}I. V. Basawa and J. Zhou.
\newblock Non-Gaussian bifurcating models and quasi-likelihood estimation.
\newblock {\em J. Appl. Probab. 41A (2004), 55-64.}

\bibitem{BH99}I. V. Basawa and R. M. Huggins.
\newblock Extensions of the bifurcating autoregressive model for cell lineage
studies.
\newblock {\em J. Appl. Probab. 36, 4 (1999), 1225-1233.}

\bibitem{BH00} I. V. Basawa. and  R. M. Huggins.
\newblock Inference for the extended bifurcating autoregressive model for cell
lineage studies.
\newblock {\em Aust. N. Z. J. Stat. 42, 4 (2000), 423-432.}

\bibitem{BHY09} I. V. Basawa, S. Y. Hwang and I. K. Yeo.
\newblock Local asymptotic normality for bifurcating autoregressive processes
and related asymptotic inference.
\newblock {\em Statistical Methodology 6 (2009), 61-69.}

\bibitem{BZ105} I. V. Basawa. J. Zhou.
\newblock Least-squares estimation for bifurcating autoregressive
processes.
\newblock {\em Statist. Probab. Lett. 74, 1 (2005), 77-88.}

\bibitem{BZ205}I. W. Basawa and J. Zhou.
\newblock Maximum likelihood estimation for a first-order bifurcating
autoregressive process with exponential errors.
\newblock {\em J. Time Ser. Anal. 26, 6 (2005), 825-842.}

\bibitem{BSG09} B. Bercu, B. de Saporta and A. G\'egout-Petit.
\newblock Asymtotic analysis for bifurcating autoregressive processes via martingale
approach.
\newblock {\em Electron. J. Probab. 14 (2009), no. 87, 2492-2526 .}

\bibitem{BeTo08}B. Bercu and A. Touati.
\newblock Exponential inequalities for self-normalized martingales with applications.
\newblock {\em Ann. Appl. Probab. 18 (2008), no. 5, 1848-1869}

\bibitem{BDG11} V. Bitseki Penda, H. Djellout and A. Guillin.
\newblock Deviation inequalities, Moderate deviations and some limit theorems for bifurcating Markov chains with application.
\newblock {\em  arXiv:1111.7303}

\bibitem{CS86}R. Cowan and R. G.  Staudte.
\newblock The bifurcating autoregressive model in cell lineage studies.
\newblock {\em Biometrics 42 (1986), 769-783.}

\bibitem{DM10}J. F.  Delmas and L. Marsalle.
\newblock Detection of cellular aging in a Galton-Watson process.
\newblock {\em  Stoch. Process. and Appl., 120 : 2495-2519, 2010.}


\bibitem{Dem96}A. Dembo.
\newblock Moderate deviations for martingales with bounded jumps.
\newblock {\em  Electron. Comm. Probab.  1  (1996), no. 3, 11-17.}

\bibitem{DemZei98}A. Dembo and O. Zeitouni.
\newblock (1998) Large Deviations Techniques
and Applications,
\newblock{\em 2nd Ed. (Springer, New York).}

\bibitem{DjGuWu06} H. Djellout, A. Guillin  and L. Wu.
\newblock Moderate deviations of empirical periodogram and non-linear functionals of moving average processes.
\newblock {\em   Ann. Inst. H. PoincarŽ Probab. Statist. 42 (2006), no. 4, 393-416. }

\bibitem{Dj02}H. Djellout.
\newblock Moderate deviations for martingale differences and applications to $\phi$-mixing sequences
\newblock {\em  Stoch. Stoch. Rep.73 (2002),1-2, 37-63.}

\bibitem{DjGu01}H. Djellout and A. Guillin.
\newblock Large and moderate deviations for moving average processes.
\newblock {\em Ann. Fac. Sci. Toulouse Math. (6)  10  (2001),  no. 1, 23-31.}

\bibitem{GoLe07}N. Gozlan and C. LŽonard.
\newblock  A large deviation approach to some transportation cost inequalities.
\newblock {\em Probab. Theory Related Fields 139 (2007), no. 1-2, 235-283.}

\bibitem{Go06}N. Gozlan.
\newblock Integral criteria for transportation-cost inequalities.
\newblock {\em  Electron. Comm. Probab. 11 (2006), 64-77.}


\bibitem{G07}J. Guyon.
\newblock  Limit theorems for bifurcating Markov chains. Application to the detection of cellular aging.
\newblock {\em  Ann. Appl. Probab. 17 (2007), no. 5-6, 1538-1569.}

\bibitem{Le01}M.  Ledoux.
\newblock The concentration of measure phenomenon.
\newblock {\em Mathematical Surveys and Monographs, 89. American Mathematical Society, Providence, RI, 2001.}

\bibitem{Ma07}P. Massart.
\newblock  Concentration inequalities and model selection.
\newblock {\em Lecture Notes in Mathematics, 1896. Springer, Berlin, 2007.}

\bibitem{Puha97}A. Puhalskii.
\newblock Large deviations of semimartingales: a maxingale problem
approach. I. Limits as solutions to a maxingale problem.
\newblock {\em Stochastics Stochastics Rep. 61 (1997),  no. 3-4, 141-243.}

\bibitem{PeTzQu09}V. H.  de la Pe$\overset{\sim}{\rm n}$a; T. L.  Lai  and Qi-Man Shao.
\newblock  Self-normalized processes. Limit theory and statistical applications.
\newblock {\em Probability and its Applications (New York). Springer-Verlag, Berlin, 2009. 275 pp.}



\bibitem{SGM11}B. de Saporta, A. G\'egout-Petit and L. Marsalle
\newblock  Parameters estimation for asymmetric bifurcating autoregressive
processes with missing data.
\newblock {\em Electronic Journal of Statistics, vol. 5 (2011) 1313-1353.}


\bibitem{SGM12}B. de Saporta, A. G\'egout-Petit and L. Marsalle
\newblock Asymmetry tests for Bifurcating Auto-Regressive Processes with missing data
\newblock {\em arXiv:1112.3745.}

\bibitem{Wei87}C. Z. Wei.
\newblock Adaptive prediction by least squares predictors in stochastic
regression models with applications to time series.
\newblock {\em Ann. Statist. 15, 4 (1987), 1667-1682.}


\bibitem{Wor01a}J. Worms.
\newblock Moderate deviations of some dependent variables. I. Martingales.
\newblock {\em Math. Methods Statist.10 (2001),no. 1, 38-72.}

\bibitem{Wor01b}J. Worms.
\newblock Moderate deviations of some dependent variables. II. Some kernel
estimators.
\newblock {\em Math. Methods Statist.10 (2001),no. 2, 161-193.}

\bibitem{Wor01c}J. Worms.
\newblock{\em  Principes de d\'eviations mod\'er\'ees pour des martingales et
applications statistiques.}
\newblock Th\`ese de Doctorat \`a l'Universit\'e Marne-la-Vall\'ee, 2000.

\bibitem{Wor99d}J. Worms.
\newblock  Moderate deviations for stable Markov chains and regression models.
\newblock {\em Electron. J. Probab. 4 (1999), no. 8, 28 pp.}




\end{thebibliography}
\end{document}